\documentclass{amsart}
\usepackage{amsmath}
\usepackage{amsthm}
\usepackage{amssymb}

\newtheorem{thm}{Theorem}
\newtheorem{lem}[thm]{Lemma}
\newtheorem{cor}[thm]{Corollary}
\newtheorem{defi}[thm]{Definition}
\newtheorem{prop}[thm]{Proposition}
\newtheorem{rk}[thm]{Remark}

\newcommand{\rr}{{\mathbb{R}}}
\newcommand{\rd}{{\rr^3}}

\newcommand{\e}{\varepsilon}
\newcommand{\indiq}{\hbox{\rm 1}{\hskip -2.8 pt}\hbox{\rm I}}
\newcommand{\Id}{\mathbf{I}_3}
\newcommand{\E}{\mathbb{E}}
\newcommand{\diver}{\mathrm {div}}
\newcommand{\Tr}{\mathrm{Tr}}

\newcommand{\intot}{\int_0^t }
\newcommand{\raa}{a^{1/2}}
\newcommand{\intrd}{\int_{\rr^3}}
\newcommand{\cA}{{\mathcal A}}
\newcommand{\cM}{{\mathcal M}}
\newcommand{\cE}{{\mathcal E}}
\newcommand{\tcE}{\widetilde{\mathcal E}}
\newcommand{\cL}{{\mathcal L}}
\newcommand{\cN}{{\mathcal N}}
\newcommand{\cH}{{\mathcal H}}
\newcommand{\cP}{{\mathcal P}}
\newcommand{\cS}{{\mathcal S}}
\newcommand{\cU}{{\mathcal U}}
\newcommand{\cW}{{\mathcal W}}
\newcommand{\bL}{{\mathbf L}}
\newcommand{\tB}{\tilde B}
\newcommand{\tU}{\tilde U}
\newcommand{\tV}{\tilde V}
\newcommand{\hB}{\hat B}
\newcommand{\tbeta}{\tilde \beta}
\newcommand{\lps}{\langle \!\langle}
\newcommand{\rps}{\rangle\!\rangle}

\newcommand{\vip}{\vskip.13cm}

\newcommand{\poubelle}[1]{}

\begin{document}

\title[Kac's particle system for the Landau equation]
{From a Kac-like particle system to the Landau equation 
for hard potentials and Maxwell molecules}

\author{Nicolas Fournier}
\author{Arnaud Guillin}

\address{N. Fournier: Laboratoire de Probabilit\'es et Mod\`eles al\'eatoires, 
UMR 7599, UPMC, Case 188,
4 pl. Jussieu, F-75252 Paris Cedex 5, France.}

\email{nicolas.fournier@upmc.fr}

\address{A. Guillin : Laboratoire de Math\'ematiques, UMR 6620, Universit\'e Blaise Pascal, Av. des landais,
63177 Aubi\`ere cedex, France}

\email{arnaud.guillin@math.univ-bpclermont.fr}



\begin{abstract}
We prove a quantitative result of convergence of a conservative stochastic
particle system to the solution
of the homogeneous Landau equation for hard potentials. 
There are two main difficulties: 
(i) the known stability results for this class of Landau equations 
concern {\it regular} solutions and seem difficult to extend to study the rate of convergence 
of some empirical measures; 
(ii) the conservativeness of the particle system is an obstacle for (approximate) independence.
To overcome (i), we prove a new stability result for the Landau equation for hard potentials
concerning very general measure solutions.
Due to (ii), we have to couple, our particle system with some 
{\it non independent} nonlinear processes,
of which the law solves, in some sense, the Landau equation. We then prove that these nonlinear processes
are not so far from being independent.
Using finally some ideas of Rousset \cite{r}, we show that in the case
of Maxwell molecules, the convergence of the particle system is uniform in time.
%
%\vip
%
%\vip
%
%\begin{center}{\bf D'un syst\`eme de particules de type Kac \`a l'\'equation de Landau pour des potentiels durs et des mol\'ecules Maxwelliennes}\end{center}
%
%\vip
%
%\noindent{\sc R\'esum\'e}. Nous prouvons des r\'esultats quantitatifs de convergence d'un syst\`eme de particules conservatif vers la solution de l\'equation de Landau homog\`ene pour des potentiels durs. Il y a deux principales diffixult\'es : (i) les r\'esultats connus de stabilit\'e pour ces \'equations de Landau concernent des solutions r\'eguli\`eres et paraissent difficiles \`a \'etendre pour \'etudier le taux de convergence pour des mesures empiriques ; (ii) le fait que le syst\`eme de particule soit conservatif est un obstacle pour obtenir de l'ind\'ependance (m\^eme approch\'ee). Pour r\'esoudre (i), nous \'etablissons de nouveaux r\'esultats de stabilit\'e pour l\'equation de Landau pour des potentiels durs et des solutions de type mesure tr\`es g\'en\'erales. Pour le point (ii), nous devons coupler notre syst\`eme de particules avec des processus non lin\'eaires {\it non ind\'ependants} pour lesquels la loi r\'esoud en un certain sens l\'equation de Landau. Nous prouvons ensuite que ces processus non lin\'eaires ne sont pas tr\`es loin d'\^etre ind\'ependants. Finalement, et en utilisant des id\'ees de Rousset \cite{r}, nous montrons que dans le cas des mol\'ecules maxwelliennes, la convergence du syst\`eme de particules est uniforme en temps.
\end{abstract}

\maketitle

\section{Introduction and main results}

\subsection{The Landau equation}

The homogeneous Landau equation reads
\begin{align} \label{HL3D}
\partial_t f_t(v) = \frac 1 2 
\diver_v \Big( \intrd a(v-v_*)[ f_t(v_*) \nabla f_t(v) - f_t(v) \nabla f_t(v_*)
]\,dv_* \Big).
\end{align}
The unknown $f_t:\rr^3\mapsto\rr$ stands 
for the velocity-distribution in a plasma and the initial condition $f_0$ is given.
We denote by $S_3^+$ the set of symmetric nonnegative $3\times 3$ matrices.
The function $a:\rr^3\mapsto S_3^+$ is given, for some $\gamma \in [-3,1]$, by 
\[
a(v) = |v|^{2+\gamma} \Pi_{v^\perp}, \quad \hbox{where} \quad \Pi_{v^\perp}=  \Id - \frac{v \otimes v}{|v|^2}
\]
is the projection matrix onto $v^\perp$. The only physically relevant case is
$\gamma=-3$, which corresponds to a Coulomb interaction. However, the other cases are interesting
mathematically and numerically. In particular, the Landau equation can be seen as an approximation 
of the Boltzmann 
equation in the asymptotic of {\it grazing collisions}, as rigorously 
shown by Villani \cite{v:nc} for all values of
$\gamma \in [-3,1]$.
We are concerned here with Maxwell molecules ($\gamma=0$) and hard potentials ($\gamma\in (0,1]$).
The well-posedeness, regularization properties and large-time behavior of the Landau equation
have been studied in great details by Villani \cite{v:max} for Maxwell molecules and by
Desvillettes and Villani \cite{dv,dv2} for hard potentials.
We finally refer to the
long reviews paper of Villani \cite{v:h} and Alexandre \cite{a} on the 
Boltzmann and Landau models.

\subsection{Notation}

We denote by $\cP(\rr^3)$ the set of probability measures on $\rd$.
When  $f\in\cP(\rr^3)$ has a density, we also denote by $f\in L^1(\rd)$ this density.
For $q>0$, we set $\cP_q(\rd)=\{f\in\cP(\rd)\;:\;m_q(f)<\infty\}$, where $m_q(f)=\intrd |v|^q f(dv)<\infty$.
For $\alpha>0$ and $f \in \cP(\rd)$, we put $\cE_\alpha(f)=\intrd \exp(|v|^\alpha) f(dv)$.
The entropy of $f\in \cP(\rd)$ is defined by $H(f)=\intrd f(v)\log f(v) dv$ if $f$ has a density
and by $H(f)=\infty$ else.

\vip

We will use the Wasserstein distance
defined as follows.
For $f,g\in\cP_2(\rd)$, we introduce $\cH(f,g) = \bigl\{    R \in \cP(\rr^3\times\rr^3) \; : \; 
R  \text{ has marginals } f \text{ and } g \bigr\}$ and we set 
$$
\cW_2(f,g)= \inf \Big\{ \Big(\int_{\rd\times\rd} |v-w|^2 R(dv,dw)\Big)^{1/2} \; : \; R \in \cH(f,g)   \Big\}.
$$
See Villani \cite{v:t} for many details on this distance.

\vip

We also define, for $v\in \rd$,
$$
b(v)=\diver \; a(v)=-2 |v|^\gamma v \quad \hbox{and}\quad \sigma(v)=[a(v)]^{1/2}=|v|^{1+\gamma/2}\Pi_{v\perp}.
$$
For $f\in \cP(\rd)$ and $v\in \rd$, we set
\begin{gather*}
b(f,v):=\intrd b(v-v_*)f(dv_*),\quad a(f,v):=\intrd a(v-v_*)f(dv_*), \quad
\raa(f,v):=\Big[a(f,v)\Big]^{1/2}.
\end{gather*}
More generally, we will write $\varphi(f,v)=\intrd \varphi(v-v_*)f(dv_*)$ when $\varphi:\rd\mapsto\rr$.
We emphasize that $\raa(f,v)$ is $[a(f,v)]^{1/2}$ and is not
$\intrd \raa(v-v_*)f(dv_*)$.

\vip

Finally, for $A$ and $B$ two $3\times 3$ matrices, we put
$\| A \|^2 = \Tr (AA^*)$ and $\lps A,B \rps=\Tr (A B^*)$.

\subsection{Well-posedness}
We will use the following notion of weak solutions.

\begin{defi}\label{ws}
Let $\gamma \in [0,1]$. We say that $f=(f_t)_{t\geq 0}$ is a weak solution to \eqref{HL3D} if it 
belongs to $L^\infty_{loc}([0,\infty),\cP_{2+\gamma}(\rd))$ and if 
for all $\varphi\in C^2_b(\rr^3)$, all $t\geq 0$,
\begin{align}\label{wf}
\intrd \varphi(v)f_t(dv) = \intrd \varphi(v)f_0(dv) + \intot \intrd \intrd L\varphi(v,v_*) 
f_s(dv)f_s(dv_*) ds,
\end{align}
where
$$
L\varphi(v,v_*):= \frac 1 2 \sum_{k,l=1}^3 a_{kl}(v-v_*)\partial^2_{kl}\varphi(v)+ \sum_{k=1}^3
b_{k}(v-v_*)\partial_{k}\varphi(v).
$$
A weak solution $f$ is {\rm conservative} if it conserves momentum and energy, that is
$\intrd vf_t(dv)=\intrd v f_0(dv)$ and
$m_2(f_t)=m_2(f_0)$ for all $t\geq 0$.
\end{defi}

An important remark is that $|L\varphi(v,v_*)|\leq C_\varphi (1+|v|+|v_*|)^{2+\gamma}$ for $\varphi\in C^2_b(\rr^3)$
and since $f\in L^\infty_{loc}([0,\infty),\cP_{2+\gamma}(\rd))$, every term makes
sense in \eqref{wf}. 
Our first result concerns well-posedness and stability.

\begin{thm}\label{mr1}
(i) If $\gamma=0$, then for any  $f_0 \in \cP_2(\rd)$, \eqref{HL3D} has a unique weak solution 
$f=(f_t)_{t\geq 0}$ starting from $f_0$. This solution is conservative.
If moreover $H(f_0)<\infty$, then $H(f_t)\leq H(f_0)$ for all $t\geq 0$.
If $f_0 \in \cP_q(\rd)$ for some $q>2$, then $\sup_{[0,\infty)} m_q(f_t)<\infty$.
Finally, for any other weak solution $g=(g_t)_{t\geq 0}$ to \eqref{HL3D}, it holds that $\cW_2(f_t,g_t)\leq
\cW_2(f_0,g_0)$ for all $t\geq 0$.

\vip

(ii) If $\gamma \in (0,1]$, consider $f_0\in \cP_2(\rd)$ with $\cE_{\alpha}(f_0)<\infty$
for some $\alpha \in (\gamma,2)$. 
Then \eqref{HL3D} has a unique weak solution $f=(f_t)_{t\geq 0}$ starting from $f_0$. 
Moreover, this solution is conservative and $\sup_{t\geq 0} \cE_{\alpha}(f_t)<\infty$.
If $H(f_0)<\infty$, then $H(f_t)\leq H(f_0)$ for all $t\geq 0$.
Finally, for all $\eta\in(0,1)$, all $T>0$ and 
any other weak solution to $g=(g_t)_{t\geq 0}$ to \eqref{HL3D}, it holds that $\sup_{[0,T]} \cW_2(f_t,g_t)\leq
C_{\eta,T} (\cW_2(f_0,g_0))^{1-\eta}$, the constant $C_{\eta,T}$ depending only on 
$\eta,T,\gamma,\alpha$ and on (upperbounds of) $\cE_{\alpha}(f_0)$ and $\sup_{[0,T]}m_{2+\gamma}(g_t)$.
\end{thm}

Point (i) is well-known folklore, even if we found no precise reference for all the claims of the statement. 
The well-posedness, propagation of moments and entropy dissipation has been checked 
by Villani \cite{v:max} when $f_0$ has a density and the well-posedness when
$f_0 \in \cP_2(\rd)$ has been established by Gu\'erin \cite{g}. The noticeable fact that $\cW_2$
decreases along solutions was discovered by Tanaka \cite{t} 
for the Boltzmann equation for Maxwell
molecules, see also Carrapatoso \cite[Lemma 4.15]{c}.

\vip

Similarly, the existence part in point (ii) is more or less standard: the well-posedness, propagation
of moments and entropy dissipation can be found in \cite{dv} when $H(f_0)<\infty$, but
$H(f_0)<\infty$ is mainly assumed for simplicity. The propagation of exponential moments
seems to be new, but far from surprising: it is well-known (and more complicated)
for the Boltzmann equation for hard potentials,
as was discovered by Bobylev \cite{b}, see also Alonso {\it et al.} \cite{acgm}.

\vip

On the contrary, the uniqueness/stability part in point (ii) seems to be new and rather interesting.
As far as we know, the best available 
uniqueness result is the one of Desvillettes and Villani \cite[Theorem 7]{dv},
where $f_0\in\cP_2(\rd)$ is assumed to have a density satisfying 
$\intrd f_0^2(v)(1+|v|^{s})dv <\infty$ for some $s>15+5\gamma$.
Thus, we assume much less regularity, but much more localization. Furthermore, our
stability result holds {\it in the class of all weak solutions}.
Actually, a stability result 
in the class of all weak solutions (at least with finite entropy)
can also be derived using the ideas of Desvillettes and Villani,
but this would use the regularization properties of the equation which guarantee that {\it any} weak solution
is smooth.
On the contrary, we use no such regularization. This is crucial for propagation of chaos,
since then the approximate solution consists of empirical measures which, 
by nature, are not smooth. Similarly, it is very important for us that the stability result does not involve
any exponential moment of the second solution $g$, because we are not able to propagate
the exponential moments of our particle system.

\subsection{The particle system}
We now introduce an approximating particle system, in the spirit of Kac \cite{Kac1956}, who was dealing
with the Boltzmann equation.
As shown by Carrapatoso \cite{c} when $\gamma=0$, this system can be 
derived from Kac's system in the asymptotic of grazing collisions.

\vip

We fix $N\geq 2$ and consider an exchangeable $(\rd)^N$-valued random variable $(V^{i,N}_0)_{i=1,\dots,N}$,
independent of a family $(B^{ij}_t)_{1\leq i < j < N, t\geq 0}$
of i.i.d.  3D Brownian motions. For $1\leq j<i\leq N$, we set $B^{ij}_t=-B^{ji}_t$.
We also put $B^{ii}_t=0$ for all $i=1,\dots,N$ and we 
consider the system
\begin{align}\label{ps}
V^{i,N}_t\!=\!V^{i,N}_0 \!+\! \frac 1 N \sum_{j=1}^N \intot b(V^{i,N}_s-V^{j,N}_s)ds \!+\! 
\frac 1 {\sqrt N} \sum_{j=1}^N \intot \sigma(V^{i,N}_s-V^{j,N}_s)dB^{ij}_s,\quad i=1,\dots,N.
\end{align}

\begin{prop}\label{pswp}
Fix $\gamma \in [0,1]$ and $N\geq 2$.
The system \eqref{ps}
has a pathwise unique solution $(V_t^{i,N})_{i=1,\dots,N,t\geq 0}$, which is furthermore exchangeable.
The system is conservative: a.s., for all $t\geq 0$, it holds that $\sum_1^N V^{i,N}_t=\sum_1^N V^{i,N}_0$
and $\sum_1^N |V^{i,N}_t|^2=\sum_1^N |V^{i,N}_0|^2$.
\end{prop}

In \cite{fgm}, Fontbona, Gu\'erin and M\'el\'eard consider, when $\gamma=0$, the same system of equations,
but with a fully i.i.d. family $(B^{ij}_t)_{1\leq i , j \leq N, t\geq 0}$ of Brownian motions.
Such a system also approximates the Landau equation, but is not conservative 
(one only has $\E[\sum_1^N V^{i,N}_t]=\E[\sum_1^N V^{i,N}_0]$
and $\E[\sum_1^N |V^{i,N}_t|^2]=\E[\sum_1^N |V^{i,N}_0|^2]$) and thus physically less
relevant.

\subsection{Propagation of chaos}
The main result of the paper is the following.

\begin{thm}\label{mr2}
Fix $\gamma \in [0,1]$ and $f_0 \in \cP_2(\rd)$.
If $\gamma \in (0,1]$, assume moreover that $\cE_{\alpha}(f_0)<\infty$ for some $\alpha \in(\gamma,2)$.
Consider the unique weak solution $(f_t)_{t\geq 0}$ to \eqref{HL3D} built in Theorem \ref{mr1}.
For each $N\geq 2$, consider an exchangeable $(\rd)^N$-valued random variable $(V^{i,N}_0)_{i=1,\dots,N}$
and the corresponding unique solution $(V_t^{i,N})_{i=1,\dots,N,t\geq 0}$ to \eqref{ps}. Set
$\mu^N_t=N^{-1}\sum_1^N \delta_{V^{i,N}_t}$. Assume that for all $p\geq 2$,
$M_p:=m_p(f_0)+\sup_{N\geq 2}\E[|V^{1,N}_0|^p]<\infty$.

\vip

(i) If $\gamma=0$, then for all $\eta \in (0,1)$, there is a constant $C_\eta$ depending
only on $\eta$, on (some upperbounds of) $\{M_p,p\geq 2\}$ and on (some upperbound of) $H(f_0)$ 
when $H(f_0)<\infty$ such that 
$$
\sup_{[0,\infty)}\E[\cW_2^2(\mu^N_t,f_t)]\leq \begin{cases}
C_\eta(\E[\cW_2^2(\mu^N_0,f_0)]+N^{-1/4})^{1-\eta} & \hbox{in general,}\\
C_\eta(\E[\cW_2^2(\mu^N_0,f_0)]+ N^{-1/3})^{1-\eta} & \hbox{if } H(f_0)<\infty.
\end{cases}
$$

(ii) If $\gamma \in (0,1]$, then for all $T>0$, all $\eta \in (0,1)$, there is a constant $C_{\eta,T}$ depending
only on $\eta$, $T$, $\gamma$, $\alpha$, on (some upperbounds of) $\cE_\alpha(f_0)$ and $\{M_p,p\geq 2\}$
and on  (some upperbound of) $H(f_0)$ when  $H(f_0)<\infty$ such that 
$$
\sup_{[0,T]}\E[\cW_2^2(\mu^N_t,f_t)]\leq \begin{cases}
C_{\eta,T} (\E[\cW_2^2(\mu^N_0,f_0)]+N^{-1/4})^{1-\eta} & \hbox{in general,}\\
C_{\eta,T} (\E[\cW_2^2(\mu^N_0,f_0)]+N^{-1/3})^{1-\eta}& \hbox{if } H(f_0)<\infty.
\end{cases}
$$
\end{thm}

If $(V_0^{i,N})_{i=1,\dots,N}\sim f_0^{\otimes N}$, then we know from \cite[Theorem 1]{fgui}
that $\E[\cW_2^2(\mu^N_0,f_0)]\leq C N^{-1/2}$ and that $N^{-1/2}$ 
is generally the best rate we can hope for
when comparing an empirical measure of an i.i.d. sample to the common distribution.
Here we obtain a rate in $N^{-1/3}$ (or $N^{-1/4}$ without entropy), up to an arbitrarily small loss,
which is not so bad.
Let us finally mention that in point (i), the time uniformity really uses that we are in dimension $d>2$.

\subsection{References on propagation of chaos}

Showing the convergence of a toy particle system to the Boltzmann equation
was proposed by Kac \cite{Kac1956} as a step to its rigorous derivation.
He called {\it propagation of chaos} such a convergence.
Getting some {\it uniform in time} convergence is quite relevant,
since then the large time behavior of the PDE indeed describes that of the particle system.
Another important motivation is the numerical resolution of the Boltzmann equation without cutoff:
indeed, it may be relevant to replace {\it grazing} collisions by a diffusive Landau-like term.
Choosing the right threshold level requires to know quite well the rates of convergence.
See \cite{fgod} for a complete study, in this spirit, of the 1D Kac equation.

\vip

To our knowledge, the only result {\it directly} comparable to Theorem \ref{mr2} is the one
of Carrapatoso \cite[Theorem 4.2]{c} which concerns Maxwell molecules ($\gamma=0$): he obtains
(under some different conditions on $f_0$),
a uniform in time rate of convergence in (almost) $N^{-1/972}$ for another distance, strictly controlled by
$\sup_{[0,\infty)}\E[\cW_2^2(\mu^N_t,f_t)]^{1/2}$, which we can bound by (almost) $N^{-1/6}$.

\vip

Concerning the non-conservative particle system approximating the Landau equation,  
Maxwell molecules have been studied by Fontbona, Gu\'erin and M\'el\'eard \cite{fgm}
(there it is proved that $\sup_{[0,T]}\E[\cW_2^2(\mu^N_t,f_t)] \leq C_TN^{-2/7}$), see also \cite{fou1}.
Moderately soft potentials are investigated in the companion paper \cite{fh} 
($\sup_{[0,T]}\E[\cW_2^2(\mu^N_t,f_t)] \leq C_TN^{-1/2}$ when $\gamma \in (-1/4,0)$,
a less good rate when $\gamma \in (-1,-3/4]$ and a convergence without rate when $\gamma \in (-2,-1]$).
As compared to \cite{fh}, the present situation is simpler (because hard potentials are rather easier
than soft potentials) but more complicated (because we study the conservative particle system).

\vip

Sznitman \cite{s} was the first to prove
the convergence (without rate) of Kac's conservative particle system to the Boltzmann equation
for hard spheres ($\gamma=1$).
Some recent progresses have been made by Mischler and Mouhot \cite{mm} 
(from which \cite{c} is inspired) where, using an abstract and purely analytic method, 
a uniform in time quantitative convergence of Kac's particle system was derived, for the Boltzmann
equation for Maxwell molecules ($\gamma=0$, with a rate in $N^{-\e}$ for some very small $\e$)
and hard spheres ($\gamma=1$, with a rate in $(\log N)^{-\e}$ for some very small $\e$). 
Even if these rates are clearly far from being sharp, these
results are impressive. However, the method
uses some smoothness of the solution $(f_t)_{t\geq 0}$ with respect to $f_0$
(something like one or two derivatives, in some sense, required), which is closely related
to uniqueness/stability theory. Such a theory is completely understood only for Maxwell molecules
(where the kinetic cross section is constant) and hard spheres (where the angular cross section
is integrable).
Finally, let us mention the paper of Cortez and Fontbona \cite{cf}, who considered
the simplest model (the 1D Kac equation), but who obtained by coupling methods
a {\it good} rate of convergence (although probably not optimal, in $N^{-1/3}$) 
for a {\it conservative} particle system.
These authors told us that, putting together the ideas
of \cite{cf} and of \cite{fmi}, they are now treating the case of Kac's conservative system
for the $3D$ Boltzmann equation for hard potentials.

\subsection{Scheme of the proofs}\label{ideas}

Interpreting a solution $(f_t)_{t\geq 0}$ to a kinetic equation in terms 
of the time-marginals of a 3D process $(V_t)_{t\geq 0}$ solving some {\it nonlinear} Poisson
SDE was initiated by Tanaka \cite{t} for the Boltzmann equation for Maxwell molecules.
Roughly, $(V_t)_{t\geq 0}$ represents the time-evolution of the velocity of a {\it typical} particle.
A similar process was proposed by Gu\'erin \cite{g} for the Landau equation,
with a white noise-driven SDE.
Here and in \cite{fh}, we rather use a Brownian SDE.
We show that for any weak solution $(f_t)_{t\geq 0}$ and for $V_0\sim f_0$, the SDE
$V_t=V_0+\intot [b(f_s,V_s)ds+\raa(f_s,V_s)dB_s]$ is well-posed and $V_t\sim f_t$ for all $t\geq 0$.
We call  $(V_t)_{t\geq 0}$ a $(f_t)_{t\geq 0}$-Landau process.
To prove uniqueness/stability, we will consider two weak solutions $(f_t)_{t\geq 0}$ and
$(g_t)_{t\geq 0}$ and we will
couple a $(f_t)_{t\geq 0}$-Landau process and a $(g_t)_{t\geq 0}$-Landau process
in such a way that they remain as close as possible.
Using the same Brownian motion for both processes, sometimes called synchronous coupling, does not provide sufficiently good estimates.
We will use a finer coupling, based on some ideas of Givens and Shortt \cite{gs} about the optimal coupling of
(multidimensional) Gaussian random variables (for  $\cW_2$ distance).
Such a finer coupling is crucial, in particular to obtain a stability result that requires 
exponential moments of only one of the two solutions. As already mentioned,
this is important because we are not able to propagate exponential moments of the particle
system.

\vip

Similarly, we will finely couple our particle system $(V^{i,N})_{i=1,\dots,N}$ with a family $(W^{i,N})_{i=1,\dots,N}$
of $(f_t)_{t\geq 0}$-Landau processes.
The conservativeness of our particle system implies that the family $(W^{i,N}_t)_{i=1,\dots,N}$
is not independent. But we will use a second coupling to show that for 
$1<<K<<N$,  $(W^{i,N}_t)_{i=1,\dots,K}$ are approximately independent.
The idea of using two couplings is already present in the
paper by Cortez and Fontbona \cite{cf}.
\vip

The time uniformity we obtain in the case of Maxwell molecules relies on a recent noticeable argument of Rousset
\cite{r} for the Boltzmann equation.
For two solutions $(f_t)_{t\geq 0}$ and $(g_t)_{t\geq 0}$, 
Tanaka's theorem \cite{t} tells us (roughly) that $(d/dt)\cW_2(f_t,g_t)\leq 0$.
Rousset manages to prove, in dimension $d\geq 3$, something like $(d/dt)\cW_2(f_t,g_t)\leq 
-\kappa_\e \cW_2^{1+\e}(f_t,g_t)$
for all $\e>0$. This implies
that $f_t$ tends to a unique equilibrium as $t\to\infty$ at some arbitrarily fast polynomial speed.
Much better, he gets a similar result for the particle system, uniformly in $N$.
Again, extending this strategy to the Landau equation really uses
a fine coupling with suitable different Brownian motions.

\subsection{Plan of the paper}

In the next section, we quickly prove the existence part of Theorem \ref{mr1}.
In Section \ref{secreg}, we study the regularity of 
$b,a,\sigma$ and $b(f,\cdot),a(f,\cdot),\raa(f,\cdot)$.
We prove Proposition \ref{pswp} (well-posedness of the particle system)
and the well-posedness of the Landau process in Section \ref{secch}.
Section \ref{secmain} is devoted to the proof of a central inequality,
which is used a first time in Section \ref{secwp} to prove the uniqueness/stability part
of Theorem \ref{mr1}. We next show in Section \ref{secmps} that all the moments of the particle
system propagate, uniformly in $N$ and in time. This allows us to handle the proof
of Theorem \ref{mr2} (propagation of chaos) in Section \ref{secps}, 
based on a second use of our central inequality, except the time-uniformity (when $\gamma=0$) which
is verified in Section \ref{secmax}.

\section{Existence, moments and exponential moments}\label{secex}

As we will use several times in the paper, the explicit expressions of $a$ and $b$ yield to
\begin{equation}\label{sev}
\begin{cases} \Tr\; a(v-v_*) = 2|v-v_*|^{2+\gamma}, \\ a(v-v_*)v\cdot v= |v-v_*|^\gamma(|v|^2|v_*|^2-(v\cdot v_*)^2),\\ 
b(v-v_*)\cdot v=-2|v-v_*|^\gamma(|v|^2-v\cdot v_*).
\end{cases}
\end{equation}
The existence part of Theorem \ref{mr1} is, as already mentioned,
more or less well-known.

\begin{prop}\label{ex}
Let $\gamma \in [0,1]$ be fixed and let $f_0 \in \cP_{2+2\gamma}(\rd)$. 
Then there exists a conservative 
weak solution $(f_t)_{t\geq 0}$ in the sense of Definition \ref{ws} enjoying the following
properties.

\vip

(i) If $H(f_0)<\infty$, then $H(f_t)\leq H(f_0)$ for all $t\geq 0$.

\vip

(ii) If $m_q(f_0)<\infty$ for some $q>2$, then $\sup_{[0,\infty)} m_q(f_t)\leq C_q$, for some finite
constant depending only on $\gamma,q$ and on (an upperbound of) $m_q(f_0)$.

\vip

(iii) If $\gamma \in (0,1]$ and $\cE_\alpha(f_0)<\infty$ for some $\alpha\in(0,2)$, then 
$\sup_{[0,\infty)} \cE_\alpha(f_t) \leq C_{\alpha}$, for some finite constant $C_{\alpha}$
depending only on $\alpha,\gamma$ and on (an upperbound of)  $\cE_{\alpha}(f_0)$.
\end{prop}

\begin{proof}
If $\gamma=0$, the existence (and uniqueness) of a weak solution $(f_t)_{t\geq 0}$ to \eqref{HL3D}
has been checked by Gu\'erin \cite[Corollaries 6 and 7]{g}. Point (i) is proved
by Villani \cite[Section 8]{v:max}
as well as point (ii) (see \cite[Theorem 1]{v:max}): he assumes additionally but 
does not use that $f_0 \in L^1(\rd)$.

\vip

If $\gamma \in (0,1]$ and if $f_0 \in \cP_{2+\gamma}(\rd)$ with $H(f_0)<\infty$, then we
know from Desvillettes and Villani \cite[Theorems 1 and 3]{dv} that \eqref{HL3D}
has a weak solution $(f_t)_{t\geq 0}$ satisfying points (i) and (ii).
If we only know that $f_0 \in \cP_{2+2\gamma}(\rd)$, we introduce $f_0^n=f_0\star G_n$, with
$G_n(v)=(n/2\pi)^{3/2}\exp(-n|x|^2/2)$. Then $H(f_0^n)<\infty$ and we consider a
corresponding weak solution $(f^n_t)_{t\geq 0}$, satisfying points (i) and (ii). In particular, we have
$\sup_{n\geq 1} \sup_{[0,\infty)} m_{2+2\gamma}(f^n_t) <\infty$. We thus infer from
\eqref{wf} that for all $\varphi \in C^2_b(\rd)$, 
$\sup_{n\geq 1} \sup_{[0,\infty)} |(d/dt)\intrd \varphi(v)f^n_t(dv)|<\infty$: the family
$\{(f^n_t)_{t\geq 0},n\geq 1\} \subset C([0,\infty),\cP(\rd))$ is equicontinuous
(with $\cP(\rd)$ endowed with the topology of weak convergence). We thus can find
$(f_t)_{t\geq 0} \in C([0,\infty),\cP(\rd))$ so that, up to extraction of a subsequence,
$\lim_n \sup_{[0,T]} |\intrd \varphi(v)(f^n_t-f_t)(dv)|=0$
for all $\varphi \in C_b(\rd)$ and all $T>0$. This function $(f_t)_{t\geq 0}$ 
also satisfies point (ii), because point (ii) is satisfied by  $(f^n_t)_{t\geq 0}$ uniformly in $n$. 
Thus $(f_t)_{t\geq 0} \in L^\infty([0,\infty),\cP_{2+2\gamma}(\rd))$.
Finally, it is not difficult to pass to the limit,
for each $\varphi \in C^2_b(\rd)$, each $t\geq 0$,
in the equation $\intrd \varphi(v)f_t^n(dv)=\intrd \varphi(v)f_0^n(dv) +\intot \intrd\intrd L\varphi(v,v_*)
f^n_s(dv)f^n_s(dv_*)$, to deduce that  $(f_t)_{t\geq 0}$ is a weak solution to \eqref{HL3D}:
the only difficulty is that $L\varphi$ is not bounded, but this problem is fixed using that 
$|L\varphi(v,v_*)|\leq C_\varphi(1+|v|+|v_*|)^{2+\gamma}$ and that
$\sup_{n\geq 1} \sup_{[0,\infty)} m_{2+2\gamma}(f_t+f^n_t) <\infty$.

\vip

We now assume that $\gamma \in (0,1]$, we fix $\alpha \in (0,2)$, and we 
give a formal proof of point (iii) without justifying the computations:
this probably does not prove that {\it every} weak solution propagates exponential moments,
but certainly shows that it is possible to build such weak solutions.
We consider $\varphi(v)=\exp((1+|v|^2)^{\alpha/2})$, we set $\tcE_\alpha(f)=\intrd \varphi(v)f(dv)$
and we observe that $\cE_\alpha(f)\leq \tcE_\alpha(f) \leq e \cE_\alpha(f)$.
It holds that $\partial_k \varphi(v)=\alpha v_k (1+|v|^2)^{\alpha/2-1}\varphi(v)$ and
$\partial_{kl} \varphi(v)=\alpha[(1+|v|^2)^{\alpha/2-1}\indiq_{\{k=l\}}+(\alpha-2)v_kv_l(1+|v|^2)^{\alpha/2-2}
+ \alpha v_kv_l(1+|v|^2)^{\alpha-2}]\varphi(v)$, whence
\begin{align*}
L\varphi(v,v_*)=&\frac \alpha 2\Big[  2(1+|v|^2)^{\alpha/2-1} v \cdot b(v-v_*)  
+(1+|v|^2)^{\alpha/2-1} \Tr\; a(v-v_*) \\
& \hskip3cm + \big((\alpha-2)(1+|v|^2)^{\alpha/2-2}+\alpha (1+|v|^2)^{\alpha-2}\big) a(v-v_*)v\cdot v
\Big]\varphi(v).
\end{align*}
Recalling \eqref{sev}, we find
\begin{align*}
L\varphi(v,v_*)=&\frac \alpha 2 |v-v_*|^\gamma (1+|v|^2)^{\alpha/2-2}\Big[-2(1+|v|^2)|v|^2+  2(1+|v|^2)|v_*|^2\\
& \hskip3cm+ \big((\alpha-2)+\alpha (1+|v|^2)^{\alpha/2}\big)\big(|v|^2|v_*|^2-(v\cdot v_*)^2\big)\Big]\varphi(v).
\end{align*}
Using that $|v-v_*|^\gamma \geq |v|^\gamma - |v_*|^\gamma$
and that $|v-v_*|^\gamma\leq |v|^\gamma+|v_*|^\gamma$, we deduce that
\begin{align*}
L\varphi(v,v_*) \leq&  -\alpha \Big[ (1+|v|^2)^{\alpha/2-1}|v|^{2+\gamma} 
-(1+|v|^2)^{\alpha/2-1}|v|^2 |v_*|^\gamma\Big]\varphi(v)\\
&+\frac\alpha 2(|v|^\gamma+|v_*|^\gamma)(1+|v|^2)^{\alpha/2-2}\Big[ 2(1+|v|^2)|v_*|^2\\
&\hskip3cm+ \big((\alpha-2)+\alpha (1+|v|^2)^{\alpha/2}\big)\big(|v|^2|v_*|^2-(v\cdot v_*)^2\big)\Big]\varphi(v)\\
\leq & -\alpha (1+|v|^2)^{\alpha/2-1}|v|^{2+\gamma}\varphi(v)+ C\big((1+|v|^2)^{\alpha/2}+(1+|v|^2)^{\gamma/2+\alpha-1}\big)
(1+|v_*|^{2+\gamma})\varphi(v)
\end{align*}
for some constant $C$ depending only on $\gamma,\alpha$. By the weak formulation of \eqref{HL3D}, we get
\begin{align*}
\frac{d}{dt}\tcE_\alpha(f_t) \leq&
\intrd \Big[- \alpha (1+|v|^2)^{\alpha/2-1}|v|^{2+\gamma} \\
&\hskip2cm+ C\big((1+|v|^2)^{\alpha/2}+(1+|v|^2)^{\gamma/2+\alpha-1}\big)(1+m_{2+\gamma}(f_t))\Big] \varphi(v) f_t(dv).
\end{align*}
But we know from point (ii) that $\sup_{[0,\infty)}m_{2+\gamma}(f_t)$ is bounded by some constant depending
only on $\gamma$ and $m_{2+\gamma}(f_0)$ (which is itself controlled by $\cE_\alpha(f_0)$). We end with
\begin{align*}
\frac{d}{dt}\tcE_\alpha(f_t) \leq 
\intrd \Big[- \alpha (1+|v|^2)^{\alpha/2-1}|v|^{2+\gamma}+
C(1+|v|^2)^{\alpha/2}+C(1+|v|^2)^{\gamma/2+\alpha-1}\Big]\varphi(v) f_t(dv).
\end{align*}
For large values of $|v|$, we have $(1+|v|^2)^{\alpha/2-1}|v|^{2+\gamma}\simeq |v|^{\alpha+\gamma}$
and $(1+|v|^2)^{\alpha/2} + (1+|v|^2)^{\gamma/2+\alpha-1}\simeq |v|^{\max\{\alpha,\gamma+2\alpha-2 \}}$.
But $\alpha+\gamma>\alpha$ (because $\gamma>0$) and $\alpha+\gamma> \gamma+2\alpha-2$
(because $\alpha<2$), so that we can find some constants $K,L\geq 0$ so that for all $v\in \rd$,
$$
- \alpha (1+|v|^2)^{\alpha/2-1}|v|^{2+\gamma} + C(1+|v|^2)^{\alpha/2} +C(1+|v|^2)^{\gamma/2+\alpha-1}
\leq - 1 + K\indiq_{\{|v|\leq L\}}.
$$
Consequently,
$$ 
\frac{d}{dt}\tcE_\alpha(f_t) \leq - \tcE_\alpha(f_t) +K \intrd \indiq_{\{|v|\leq L\}} \varphi(v)f_t(dv)
\leq - \tcE_\alpha(f_t) +K\varphi(L).
$$
We classically deduce that $\sup_{[0,\infty)} \tcE_\alpha(f_t) \leq \max\{ \tcE_\alpha(f_0),K\varphi(L)\}$
as desired.
\end{proof}

\section{Regularity estimates}\label{secreg}

The following estimates can be found in \cite[Lemma 11]{fou1} (with $C=1$, but with another norm). Let $S_3^+$ be the set of symmetric nonnegative $3\times 3$-matrices with real entries.

\begin{lem}\label{math}
There is a constant $C$ such that for any $A,B \in S_3^+$, 
$$
\|A^{1/2}-B^{1/2}\|\leq C\|A-B\|^{1/2} \quad \text{and} \quad 
\|A^{1/2}-B^{1/2}\|\leq C(\|A^{-1}\|\land\|B^{-1}\|)^{1/2}\|A-B\|.
$$
\end{lem}

We will sometimes need the ellipticity estimate of 
Desvillettes and Villani \cite[Proposition 4]{dv}.

\begin{lem}\label{ellip1}
Let $\gamma \in [0,1]$. For all $A>0$, there is $C_A$ depending only on $A$ and $\gamma$ 
such that for all $f\in\cP_2(\rr^3)$ satisfying 
$H(f)\leq A$ and $m_2(f)\leq A$, for all $v\in\rd$,
$\|[a(f,v)]^{-1}\| \leq C_A (1+|v|)^{-\gamma}$.
\end{lem}

We next observe that the coefficients $a$, $b$ and $\sigma$ are locally Lipschitz continuous.

\begin{lem}\label{f1}
Fix $\gamma \in [0,1]$. There is $C$ depending only on $\gamma$ such that for all $v,w \in \rd$, 
\begin{gather*}
|b(v)-b(w)|\leq C |v-w|(|v|^\gamma+|w|^\gamma), \qquad 
\|\sigma(v)-\sigma(w)\| \leq C |v-w|(|v|^{\gamma/2}+|w|^{\gamma/2}), \\
\hbox{and}\quad \|a(v)-a(w)\| \leq C |v-w|(|v|^{1+\gamma}+|w|^{1+\gamma}).
\end{gather*}
\end{lem}

\begin{proof}
Since $b(v)=-2|v|^\gamma v$, since $\sigma(v)=|v|^{\gamma/2+1}\Pi_{v^\perp}
=|v|^{\gamma/2+1}(\Id-|v|^{-2} v\otimes v)$ and since $a(v)=|v|^{2+\gamma}(\Id-|v|^{-2} v\otimes v)$,
one easily checks that $|D b(v)|\leq C |v|^\gamma$,
that $|D \sigma(v)|\leq C |v|^{\gamma/2}$ and that $|D a(v)|\leq C |v|^{1+\gamma}$, from which the results follow.
\end{proof}

Our main results are based on the use of a SDE of which we now study roughly the coefficients.

\begin{lem}\label{f2}
Fix $\gamma \in [0,1]$.  There is $C$ depending only on $\gamma$ such that 
for every $f \in \cP_{2+\gamma}(\rd)$ and every $v,w \in \rd$,

\vip

(i) $|b(f,v)| \leq C(|v|^{1+\gamma}+m_{1+\gamma}(f))$,

\vip
(ii) $|b(f,v) - b(f,w)| \leq C |v-w| (|v|^\gamma+|w|^\gamma + m_{\gamma}(f))$,
\vip
(iii) $\|a(f,v)\| \leq C (   |v|^{2+\gamma} +m_{2+\gamma}(f)  )$,
\vip
(iv) $\|a(f,v) - a(f,w)\| \leq C |v-w|(|v|^{1+\gamma}+|w|^{1+\gamma} + m_{1+\gamma}(f))$,
\vip
(v) $\|\raa(f,v)\|^2 \leq C (   |v|^{2+\gamma} +m_{2+\gamma}(f))$,
\vip
(vi) $\|\raa(f,v) - \raa(f,w)\|^2 \leq C |v-w|^2 (1+m_{2+\gamma}(f))(1+|v|^2+|w|^2)$.
\end{lem}

\begin{proof}
First, we have $|b(f,v)|\leq 2 \intrd |v-w|^{1+\gamma}f(dw) \leq C(|v|^{1+\gamma}+m_{1+\gamma}(f))$
and $\|a(f,v)\|\leq \|\raa(f,v)\|^2=\Tr \; a(f,v)=\intrd \Tr\; a(v-w) f(dw)=2 \intrd |v-w|^{2+\gamma} f(dw)
\leq  C (   |v|^{2+\gamma} +m_{2+\gamma}(f)  )$. 

\vip

Next, $|b(f,v)-b(f,w)|=|\intrd (b(v-z)-b(w-z))f(dz)|$, so that by Lemma \ref{f1},
$$
|b(f,v)-b(f,w)|
\leq C |v-w| \intrd (|v-z|^\gamma +|w-z|^\gamma )f(dz)\leq C |v-w|(|v|^\gamma+|w|^\gamma + m_{\gamma}(f)).$$
With the same arguments, one finds $\|a(f,v)-a(f,w)\|
\leq C |v-w| \intrd (|v-z|^{1+\gamma} +|w-z|^{1+\gamma} )f(dz)\leq C |v-w|(|v|^{1+\gamma}+|w|^{1+\gamma} + m_{1+\gamma}(f))$.

\vip

Point (vi) is more difficult, although probably far from being optimal.
Stroock and Varadhan \cite[Theorem 5.2.3]{SV}
state that there is $C>0$ such that for all $A: \rr^3 \mapsto S_3^+$,
$\|D(A^{1/2})\|_\infty \leq C \|D^2 A\|_{\infty}^{1/2}$, which we apply to $A(v) = (1+|v|^2)^{-\gamma/2}a(f,v)$. 
Observing that $\|a(z)\|\leq C |z|^{2+\gamma}$, that $\|D a(z)\| \leq C |z|^{1+\gamma}$ and that
$\|D^2 a(z)\| \leq C |z|^{\gamma}$, we find
\begin{align*}
\|D^2 A(v)\| \leq& C\Big[(1+|v|^2)^{-\gamma/2} \intrd |v-z|^\gamma f(dz) + (1+|v|^2)^{-\gamma/2-1/2} \intrd |v-z|^{1+\gamma} 
f(dz) \\
& \hskip6cm +  (1+|v|^2)^{-\gamma/2-1} \intrd |v-z|^{2+\gamma}f(dz) \Big]\\
\leq & C (1+m_{2+\gamma}(f)).
\end{align*}
Thus $\|D(A^{1/2})\|_\infty^2\leq C (1+m_{2+\gamma}(f))$ and
$\|(A(v))^{1/2}-(A(w))^{1/2}\|^2 \leq C  (1+m_{2+\gamma}(f))|v-w|^2$. We now write,
using that $(A(v))^{1/2}=(1+|v|^2)^{-\gamma/4}\raa(f,v)$,
\begin{align*}
\|\raa(f,v)-\raa(f,w)\|^2\leq & 2 (1+|v|^2)^{\gamma/2} 
\|(A(v))^{1/2}-(A(w))^{1/2}\|^2 \\
&+ 2(1+|w|^2)^{-\gamma/2}\big|(1+|v|^2)^{\gamma/4}- (1+|w|^2)^{\gamma/4}\big|^2 \|\raa(f,w)\|^2.
\end{align*}
Recalling (v) and using that $|(1+|w|^2)^{\gamma/4}-(1+|v|^2)^{\gamma/4}|\leq C |v-w|$,
we get 
\begin{align*}
&\|\raa(f,v)-\raa(f,w)\|^2\\
\leq & C|v-w|^2\Big[(1+|v|^2)^{\gamma/2}  (1+m_{2+\gamma}(f))
+(1+|w|^2)^{-\gamma/2}( |w|^{2+\gamma} +m_{2+\gamma}(f)) \Big].
\end{align*}
This can be bounded by $C|v-w|^2 (1+m_{2+\gamma}(f))(1+|v|^2+|w|^2)$ as desired.
\end{proof}

\section{Well-posedness of the particle system and of the Landau process}\label{secch}

We first verify that the particle system \eqref{ps} is well-posed.

\begin{proof}[Proof of Proposition \ref{pswp}]
Since $b$ and $\sigma$ are locally Lipschitz continuous by Lemma \ref{f1}, the system
classically admits a pathwise unique {\it local} solution $(V^{i,N}_t)_{i=1,\dots,N,t\in [0,\tau)}$
with $\tau=\sup_{k\geq 1} \tau_k$ and $\tau_k=\inf\{t\geq 0 \; : \; \sum_1^N |V^{i,N}_t|^2 \geq k  \}$.
We now show that a.s., $\sum_1^N V^{i,N}_t=\sum_1^N V^{i,N}_0$ and $\sum_1^N |V^{i,N}_t|^2=\sum_1^N |V^{i,N}_0|^2$
for all $t\in [0,\tau)$. This will of course imply that $\tau=\infty$ and thus end the proof.

\vip

Summing \eqref{ps} over $i=1,\dots,N$, using that $b(-x)=-b(x)$, that $\sigma(-x)=\sigma(x)$,
that $\sigma(0)=0$ 
and that $B^{ij}=-B^{ji}$ for all $i \ne j$, we immediately find that $\sum_1^N V^{i,N}_t=\sum_1^N V^{i,N}_0$
for all $t\in [0,\tau)$. 
We next apply the It\^o formula, which is licit on $[0,\tau)$, to get, using that $\sigma(x) \sigma^*(x) 
= a(x)$,
\begin{align*}
\sum_{i=1}^N |V^{i,N}_t|^2=&\sum_{i=1}^N |V^{i,N}_0|^2 + \frac 1 N \sum_{i,j=1}^N \intot
[2V^{i,N}_s \cdot b(V^{i,N}_s-V^{j,N}_s)+\Tr \; a(V^{i,N}_s-V^{j,N}_s)]ds \\
&+ \frac 2 {\sqrt N} \sum_{i,j=1}^N \intot V^{i,N}_s \cdot \sigma(V^{i,N}_s-V^{j,N}_s)dB^{ij}_s.
\end{align*}
But since $b(x)=-2|x|^\gamma x$ and $\Tr \; a(x)=2|x|^{\gamma +2}$,
\begin{align*}
&\sum_{i,j=1}^N [2V^{i,N}_s \cdot b(V^{i,N}_s-V^{j,N}_s)+\Tr \; a(V^{i,N}_s-V^{j,N}_s)]\\
=&\sum_{i,j=1}^N [(V^{i,N}_s-V^{j,N}_s)\cdot b(V^{i,N}_s-V^{j,N}_s)+\Tr\;a(V^{i,N}_s-V^{j,N}_s)]=0.
\end{align*}
Using next that $\sigma(-x)=\sigma(x)$ and that $B^{ij}=-B^{ji}$, we also have 
$$
\sum_{i,j=1}^N V^{i,N}_s \cdot \sigma(V^{i,N}_s-V^{j,N}_s)dB^{ij}_s=\sum_{1\leq i<j\leq N}(V^{i,N}_s-V^{j,N}_s) 
\cdot \sigma(V^{i,N}_s-V^{j,N}_s)dB^{ij}_s, 
$$
which a.s. vanishes because $\sigma(x)=|x|^{1+\gamma/2}\Pi_{x^\perp}$ (so that $x^* \sigma(x)=0$).
We conclude that $\sum_1^N |V^{i,N}_t|^2=\sum_1^N |V^{i,N}_0|^2$ on $[0,\tau)$, which ends the proof.
\end{proof}

We next build our Landau process.

\begin{prop}\label{f3}
Fix $\gamma \in [0,1]$ and $f=(f_t)_{t\geq 0} \in L^\infty_{loc}([0,\infty),\cP_{2+\gamma}(\rd))$, as well
as $g_0 \in \cP_2(\rd)$ and a $g_0$-distributed random variable $V_0$ independent of a 3D
Brownian motion $(B_t)_{t\geq 0}$. 

\vip

(i) The SDE $V_t=V_0+\intot b(f_s,V_s)ds + \intot \raa(f_s,V_s)dB_s$ has a pathwise unique solution.

\vip

(ii) If $f$ is a weak solution to \eqref{HL3D}
and if $g_0=f_0$, then $V_t$ is $f_t$ distributed for all $t\geq 0$.
\end{prop}

\begin{proof}
We start with point (i). Since the coefficients $v\mapsto b(f_s,v)$ and $v\mapsto \raa(f_s,v)$
are locally Lipschitz continuous (uniformly on compact time intervals) by Lemma \ref{f2}
and because $f \in  L^\infty_{loc}([0,\infty),\cP_{2+\gamma}(\rd))$ by assumption, the SDE under study classically has
a pathwise unique {\it local} solution $(V_t)_{t\in [0,\tau)}$, where $\tau=\sup_{k\geq 1} \tau_k$ and 
$\tau_k=\inf\{t\geq 0 \; : \; |V_t| \geq k  \}$. We thus only have to verify that $\tau=\infty$ a.s.
Using the It\^o formula and taking expectations, one easily checks that for all $k\geq 1$, all $t\geq 0$,
$\E[|V_{t\land \tau_k}|^2]=\E[|V_{0}|^2]+ \E [\int_0^{t\land \tau_k} \kappa(s,V_s) ds ]$,
where $\kappa(s,v)=2x\cdot b(f_s,v) + \Tr \; a(f_s,v)$.
Recalling that  $b(v)=-2|v|^\gamma v$ and $\Tr \; a(v)=2|v|^{2+\gamma}$, we find that
$\kappa(s,v) = 2 \intrd |v-w|^\gamma (|w|^2-|v|^2)f_s(dw)\leq 2 \intrd |v-w|^\gamma |w|^2f_s(dw)$.
It is not hard to deduce that $\kappa(s,v) \leq C (1+m_{2+\gamma}(f_s))(1+|v|^2)$ and then that
$$
\E[|V_{t\land \tau_k}|^2]\leq m_2(g_0)+ C\intot (1+m_{2+\gamma}(f_s))\E[1+|V_{s\land \tau_k}|^2]ds
$$
for all $t\geq 0$ and all $k\geq 1$. Since $m_{2+\gamma}(f_s)$ is locally bounded by assumption,
the Gronwall lemma implies that for all $T>0$, $C_T:=\sup_{k\geq 1} \sup_{[0,T]} \E[|V_{t\land \tau_k}|^2]<\infty$.
Hence for all $T$, $\Pr(\tau_k\leq T)= k^{-2}\E[|V_{\tau_k}|^2 \indiq_{\tau_k\leq T} ]
\leq \E[|V_{T\land \tau_k}|^2]\leq C_T k^{-2} \to 0$
as $k\to \infty$. We conclude that $\tau=\infty$ a.s.

\vip

We now prove (ii).
For $t\geq 0$ and $\varphi \in C^2_b(\rd)$, we introduce 
$\cA_t\varphi(v)=\intrd L\varphi(v,v_*) f_t(dv_*)=(1/2) \sum_{k,l=1}^3 a_{kl}(f_t,v)\partial^2_{kl}\varphi(v)+ 
\sum_{k=1}^3 b_{k}(f_t,v)\partial_{k}\varphi(v)$. Then $g_t=\cL(V_t)$ solves
\begin{equation}\label{pdelin}
\intrd \varphi(v) g_t(dv) = \intrd \varphi(v) \mu(dv) + \intot \intrd \cA_s \varphi(v) g_s(dv) ds \quad
\hbox{for all $\varphi \in C^2_c(\rd)$},
\end{equation}
with $\mu=g_0$.
But $(f_t)_{t\geq 0}$, being a weak solution to \eqref{HL3D}, also solves \eqref{pdelin} with $\mu=f_0$.
Horowitz and Karandikar \cite[Theorem B.1]{hk}, who generalize
Ethier and Kurtz \cite[Chapter 4, Theorem 7.1]{ek}, tell us that 
\eqref{pdelin} has a unique solution (for any given $\mu\in\cP(\rd)$). Since  $f_0=g_0$ by assumption, we
thus have $(f_t)_{t\geq 0}=(g_t)_{t\geq 0}$.

\vip
To apply  \cite[Theorem B.1]{hk}, we need to verify the following conditions:

\vip

(a) $C^2_c(\rd)$ is dense in $C_0(\rd)$ (the set of continuous functions vanishing at infinity) 
for the uniform convergence;

\vip

(b) for each $\varphi \in C^2_c(\rd)$, $(t,v) \mapsto \cA_t \varphi(v)$ is measurable;

\vip

(c) for each $t\geq 0$, if $\varphi \in C^2_c$ attains its maximum at $v_0$, then $\cA_t\varphi(v_0)\leq 0$;

\vip

(d) there is a countable family $\{\varphi_k\}_{k\geq 1} \subset C_c(\rd)$ such that for all $t\geq 0$,
$\{(\varphi_k,\cA_t\varphi_k)\}_{k\geq 1}$ is dense in $\{(\varphi,\cA_t\varphi),\; \varphi \in C^2_c(\rd)\}$ 
for the bounded-pointwise convergence;

\vip

(e) for any deterministic $(t_0,v_0)\in [0,\infty)\times \rd$, there exists a unique (in law) 
continuous $\rd$-valued process $(X_t)_{t\geq t_0}$ such that $X_{t_0}=v_0$ and 
for all $\varphi\in C^2_c(\rd)$, the process $\varphi(X_t)-\int_{t_0}^t \cA_s\varphi(X_s)ds$ is a martingale.

\vip

Points (a) and (b) are obvious, as well as point (c) (simply because $\nabla \varphi(v_0)=0$, 
because the Hessian $(\partial_{kl}\varphi(v_0))_{kl}$ is non-positive and because $a(f_t,v_0)$ is nonnegative).
Point (e) is equivalent to the existence and uniqueness in law,
for each $(t_0,v_0)\in [0,\infty)\times \rd$, for the SDE $V_t=v_0+\int_{t_0}^t b(f_s,V_s)ds
+\int_{t_0}^t \raa(f_s,V_s)dB_s$. If $t_0=0$, this follows from point (i) (choose $g_0=\delta_{v_0}$).
The generalization to all positive values of $t_0$ is clearly not an issue.
For (d),
consider a countable family
$\{\varphi_k\}_{k\geq 1} \subset C_c(\rd)$ so that for any $\varphi \in C^2_c(\rd)$ with, say 
Supp $\varphi\subset B(0,R)$, there is a subsequence $(k_n)_{n\geq 1}$ so that Supp $\varphi_{k_n}\subset B(0,2R)$
and $\lim_n[|\varphi_{k_n}-\varphi|_\infty+|\nabla\varphi_{k_n}-\nabla\varphi|_\infty+|D^2\varphi_{k_n}-D^2\varphi|_\infty
]=0$.
Then for each $t\geq 0$, we clearly have 
$\lim_n \|\cA_t\varphi_{k_n} - \cA_t \varphi\|_\infty=0$.
\end{proof}

\section{A central inequality}\label{secmain}

As already explained in Subsection \ref{ideas}, our uniqueness, stability and propagation of chaos
results are based on some coupling between SDEs, and using similar Brownian motions is not
sufficient to our purposes. We recall the following fact: 
the best coupling between two multidimensional Gaussian distributions
$\cN(0,\Sigma_1)$ and $\cN(0,\Sigma_2)$ does not, in general, consist in setting $X_1=\Sigma_1^{1/2}Y$ and
$X_2=\Sigma_2^{1/2}Y$ for the same $Y$ with law $\cN(0,\Id)$. As shown by Givens and Shortt
\cite{gs}, the optimal  coupling
is obtained when setting $X_1=\Sigma_1^{1/2}Y$ and $X_2=\Sigma_2^{1/2}U(\Sigma_1,\Sigma_2)Y$, where 
\begin{equation}\label{dfU}
U(\Sigma_1,\Sigma_2)=\Sigma_2^{-1/2}\Sigma_1^{-1/2}(\Sigma_1^{1/2}\Sigma_2\Sigma_1^{1/2} )^{1/2}
\end{equation}
is an orthogonal matrix. Point (i) below, proved in \cite{fh}, is an immediate
consequence of \cite{gs}.

\begin{lem} \label{tictactoc}
(i) Let $m$ be a probability measure on some measurable space $F$,
consider a pair of measurable families of $3\times 3$ matrices $(\sigma_1(x))_{x\in F}$ and $(\sigma_2(x))_{x\in F}$
and set $\Sigma_i=\int_F \sigma_i(x)\sigma_i^*(x) m(dx)$. If $\Sigma_1$ and $\Sigma_2$ are invertible,
$$
\bigl\|\Sigma_1^{1/2}-\Sigma_2^{1/2} U(\Sigma_1,\Sigma_2) \bigr\|^2 \leq \int_F \|\sigma_1(x)-\sigma_2(x)\|^2 m(dx).
$$

(ii) Let $\e\in(0,1)$. With the same notation as in (i) but without assuming that $\Sigma_1$ and $\Sigma_2$ 
are invertible, setting 
$U_\e(\Sigma_1,\Sigma_2)=U(\Sigma_1+\e \Id, \Sigma_2 + \e \Id)$,
$$
\bigl\|\Sigma_1^{1/2}-\Sigma_2^{1/2} U_\e(\Sigma_1,\Sigma_2) \bigr\|^2 \leq 
C\sqrt\e (1+\|\Sigma_1+\Sigma_2\|^{1/2}) +
\int_F \|\sigma_1(x)-\sigma_2(x)\|^2 m(dx),
$$
where $C$ is a universal constant.

\vip

(iii) For each $\e\in (0,1)$, the map $(\Sigma_1,\Sigma_2) \mapsto U_\e(\Sigma_1,\Sigma_2)$ is 
locally Lipschitz continuous on $S_3^+\times S_3^+$.
\end{lem}

Of course, we introduced $U_\e$ to avoid some technical problems, because we will generally not 
be able to control the invertibility of the matrices we will study.

\begin{proof}
Point (i) is nothing but \cite[Lemma 3.1]{fh} and point (iii) is obvious. To check (ii), we introduce
the space
$F'=F\cup\{\Delta\}$ (where $\Delta\notin F$ is some abstract point), the probability measure 
$m'=(1-\e)\indiq_F m + \e\delta_\Delta$ on $F'$, and the maps 
$\sigma_i'=(1-\e)^{-1/2}\sigma_i \indiq_F+\Id\indiq_{\{\Delta\}}$ from $F'$ to $\cM_{3\times 3}(\rr)$.
It holds that $\int_{F'} \sigma_i'(\sigma_i')^* dm'= \Sigma_i+\e \Id$, so that point (i) yields
$$
\bigl\|(\Sigma_1+\e \Id)^{1/2}-(\Sigma_2+\e \Id)^{1/2} U_\e(\Sigma_1,\Sigma_2) \bigr\|^2\! \leq \! 
\int_{F'} \|\sigma_1'(x)-\sigma_2'(x)\|^2 m'(dx)\! =\! \int_{F} \|\sigma_1(x)-\sigma_2(x)\|^2 m(dx).
$$
It then easily follows, using that $U_\e(\Sigma_1,\Sigma_2)$ is orthogonal (whence $\|U_\e(\Sigma_1,\Sigma_2)\|^2
=\Tr \; \Id =3$) and Lemma \ref{math} (which gives $\|(\Sigma_i+\e \Id)^{1/2} -\Sigma_i^{1/2}\|\leq C \sqrt \e$),
that
\begin{align*}
\bigl\|\Sigma_1^{1/2}-\Sigma_2^{1/2} U_\e(\Sigma_1,\Sigma_2) \bigr\| \leq& C\sqrt \e +
\bigl\|(\Sigma_1+\e \Id)^{1/2}-(\Sigma_2+\e \Id)^{1/2} U_\e(\Sigma_1,\Sigma_2) \bigr\| \\
\leq& C\sqrt \e + \Big( \int_{F} \|\sigma_1(x)-\sigma_2(x)\|^2 m(dx) \Big)^{1/2}.
\end{align*}
The conclusion follows: it suffices to take squares and to note that 
$\int_{F} \|\sigma_1(x)-\sigma_2(x)\|^2 m(dx)\leq 2 \int_{F} (\|\sigma_1(x)\|^2+\|\sigma_2(x)\|^2) m(dx)
=2 \Tr (\Sigma_1+\Sigma_2)\leq C \|\Sigma_1+\Sigma_2\|$.
\end{proof}

The following proposition, to be used several times for both uniqueness and propagation of chaos, plays a
central role in the paper. The $\e$ present in the statement is here only for technical reasons and may be
disregarded at first read.

\begin{prop}\label{f4}
Let $\gamma \in [0,1]$ be fixed, let $f,g \in \cP_{2+\gamma}(\rd)$ and $R \in \cH(f,g)$.
For $\e\in(0,1)$, let
\begin{align*}
\Gamma_\e(R)=\int_{\rd\times\rd} \Big(\|\raa(f,v)-\raa(g,w)&U_\e(a(f,v),a(g,v))\|^2\\
&+2(v-w)\cdot(b(f,v)-b(g,w)) \Big) R(dv,dw).
\end{align*}

(i) If $\gamma=0$, there is a universal constant $C$ such that $\Gamma_\e(R) \leq C \sqrt \e (1+m_2(f+g))^{1/2}$.

\vip

(ii) If $\gamma\in (0,1]$, then we fix $\alpha>\gamma$. 
There are some constants $\kappa>0$ and $C$ depending only on $\gamma,\alpha$, such that
for all $M>0$,
$$
\Gamma_\e(R) \leq C \sqrt \e (1+m_{2+\gamma}(f+g))^{1/2}
+ M \int_{\rd\times\rd} |v-w|^2 R(dv,dw) + C (1+m_{2+\gamma}(g)+\cE_{\alpha}(f)) e^{-\kappa M^{\alpha/\gamma}}.
$$
\end{prop}

As already mentioned, it is important
that no exponential moment of $g$ is required in (ii).

\begin{proof} We thus fix $\gamma \in [0,1]$, $f,g \in \cP_{2+\gamma}(\rd)$, $R \in \cH(f,g)$ and $\e\in(0,1)$.

\vip

{\it Step 1.}
We first verify that for all $x,y \in \rd$,
$$
\|\sigma(x)-\sigma(y)\|^2\leq 2|x|^{2+\gamma}+2|y|^{2+\gamma}-4(|x||y|)^{\gamma/2}(x\cdot y).
$$
Recall that $\sigma(x)=|x|^{1+\gamma/2}\Pi_{x^\perp}$ and that $\|\sigma(x)\|^2=\Tr\; a(x)=2|x|^{2+\gamma}$: 
we have to check that
$\lps \sigma(x),\sigma(y)\rps \geq 2(|x||y|)^{\gamma/2}(x\cdot y)$, {\it i.e.} that 
$\lps \Pi_{x^\perp},\Pi_{y^\perp}\rps \geq 2(x\cdot y)/(|x||y|)$.
A computation shows that $\Pi_{x^\perp}\Pi_{y^\perp}=\Id - |x|^{-2}xx^* - |y|^{-2}yy^* + (x.y)|x|^{-2}|y|^{-2} xy^*$ and thus
$\lps \Pi_{x^\perp},\Pi_{y^\perp}\rps=\Tr\;\Pi_{x^\perp}\Pi_{y^\perp}=1+ (x\cdot y)^2/(|x|^2|y|^2)$.
The conclusion follows.

\vip

{\it Step 2.} We fix $v$ and $w$ and we apply Lemma \ref{tictactoc}-(ii) with $F=\rd\times\rd$, with $m=R(dy,dz)$,
with $\sigma_1(y,z)=\sigma(v-y)$ and $\sigma_2(y,z)=\sigma(w-z)$.
It holds that $\int_F \sigma_1 \sigma_1^* dm = \int_{\rd\times\rd} a(v-y)R(dy,dz)=a(f,v)$
(because $\sigma(x)\sigma^*(x)=a(x)$ and $R\in\cH(f,g)$)
and $\int_F \sigma_2 \sigma_2^* dm=a(g,w)$. We thus find 
\begin{align*}
\|\raa(f,v)-\raa(g,w)U_\e(a(f,v),a(g,w))\|^2 
\leq &C \sqrt\e(1+\|a(f,v)+a(g,w)\|)^{1/2} \\
&+\int_{\rd\times\rd} \|\sigma(v-y)-\sigma(w-z)\|^2R(dy,dz).
\end{align*}
Next, it holds that 
$b(f,v)-b(g,w)=\int_{\rd\times\rd} (b(v-y)-b(w-z))R(dy,dz)$ simply because
$R \in \cH(f,g)$. Recalling finally that $\|a(f,v)\|\leq C(m_{2+\gamma}(f)+|v|^{2+\gamma})$
by Lemma \ref{f2}, we get
\begin{align*}
\Gamma_\e(R) \leq& C \sqrt\e \int_{\rd\times\rd}(1+|v|^{2+\gamma}+|w|^{2+\gamma}+ m_{2+\gamma}(f+g))^{1/2} R(dv,dw)\\
&+\int_{\rd\times\rd}\int_{\rd\times\rd} \Delta(v,y,w,z)R(dy,dz) R(dv,dw)\\
\leq &  C \sqrt\e (1+m_{2+\gamma}(f+g))^{1/2} +\int_{\rd\times\rd}\int_{\rd\times\rd} \Delta(v,y,w,z)R(dy,dz) R(dv,dw)
\end{align*}
where 
\begin{align*}
\Delta(v,y,w,z)=&\|\sigma(v-y)-\sigma(w-z)\|^2+2(v-w)\cdot(b(v-y)-b(w-z)).
\end{align*}

{\it Step 3.} The goal of this step is to check that $\Delta(v,y,w,z)=\Delta_1(v,y,w,z)+\Delta_2(v,y,w,z)$, where 
$$
\Delta_1(v,y,w,z)=(v-w+y-z)\cdot(b(v-y)-b(w-z))
$$ 
is antisymmetric (i.e. $\Delta_1(y,v,z,w)=-\Delta_1(v,y,w,z)$) and where
$$
\Delta_2(v,y,w,z)\leq \begin{cases}
0 & \hbox{ if $\gamma=0$,} \\
4 (|v-w|^2+|y-z|^2)|v-y|^\gamma & \hbox{ if $\gamma \in (0,1]$.}
\end{cases}
$$

We introduce the shortened notation $\Delta_2=\Delta_2(v,y,w,z)$, $X=v-y$ and $Y=w-z$. By definition, we have
$\Delta_2=(X-Y)\cdot(b(X)-b(Y))+\|\sigma(X)-\sigma(Y)\|^2$.
Using that $b(X)=-2|X|^\gamma X$ and Step 1, we find
\begin{align*}
\Delta_2\leq & -2(X-Y)\cdot(|X|^\gamma X-|Y|^\gamma Y) +2|X|^{2+\gamma}+2|Y|^{2+\gamma} 
-4(|X||Y|)^{\gamma/2} (X\cdot Y)\\
=& 2(X\cdot Y)(|X|^{\gamma/2}-|Y|^{\gamma/2})^2.
\end{align*}
If $\gamma=0$, this gives $\Delta_2 \leq 0$. If now $\gamma \in (0,1]$,
we use that $(x \lor y)^{1-\gamma/2}|x^{\gamma/2}-y^{\gamma/2}|\leq |x-y|$ (for $x,y\geq 0$) to write
$$
\Delta_2 \leq 2|X||Y|(|X|^{\gamma/2}-|Y|^{\gamma/2})^2 \leq 2|X||Y|(|X|\lor |Y|)^{\gamma-2}(|X|-|Y|)^2
\leq 2(|X|\land |Y|)^{\gamma}|X-Y|^2.
$$
We conclude noting that $(|X|\land |Y|)^{\gamma}\leq |X|^\gamma=|v-y|^\gamma$ and
$|X-Y|^2 \leq 2(|v-w|^2+|y-z|^2)$.

\vip

{\it Step 4.} We now observe that $L:=\int_{\rd\times\rd}\int_{\rd\times\rd} \Delta_1(v,y,w,z)R(dy,dz) R(dv,dw)=0$.
Indeed, $\Delta_1$ being antisymmetric, we have
$L=\int_{\rd\times\rd}\int_{\rd\times\rd} \Delta_1(y,v,z,w)R(dy,dz) R(dv,dw)=-L$.

\vip

{\it Step 5.} When $\gamma=0$, it suffices to gather Steps 2, 3, 4 to conclude the proof.

\vip

{\it Step 6.} Finally, gathering Steps 2, 3, 4 when $\gamma \in (0,1]$ yields
\begin{align*}
\Gamma_\e(R) \leq& C\sqrt \e(1+m_{2+\gamma}(f+g))^{1/2}  + 
4 \int_{\rd\times\rd} \int_{\rd\times\rd} \!\!\!\!\!\!(|v-w|^2+ |y-z|^2)|v-y|^\gamma R(dy,dz)R(dv,dw)\\
=&C\sqrt \e(1+m_{2+\gamma}(f+g))^{1/2}  + 
8 \int_{\rd\times\rd} \int_{\rd\times\rd} |v-w|^2 |v-y|^\gamma f(dy)R(dv,dw).
\end{align*}
For the last equality, we used a symmetry argument and that the first marginal of $R$ is $f$.
Finally, we recall that $\alpha>\gamma$ is fixed and we write, for any $M>0$,
$$
8 \int_{\rd\times\rd} \int_{\rd} |v-w|^2|v-y|^\gamma f(dy)R(dv,dw)
\leq M \int_{\rd\times\rd} |v-w|^2 R(dv,dw) + I_M,
$$
where
\begin{align*}
I_M=& 8 \int_{\rd\times\rd} \int_{\rd} |v-w|^2|v-y|^\gamma\indiq_{\{8|v-y|^\gamma \geq M\}} f(dy)R(dv,dw)\\
\leq & 16 \int_{\rd\times\rd} \int_{\rd} (|v|^2+|w|^2)(|v|^\gamma+|y|^\gamma)[\indiq_{\{|v|^\gamma \geq M/16\}} 
+ \indiq_{\{|y|^\gamma \geq M/16\}} ]f(dy)R(dv,dw).
\end{align*}
We then write, for $a>0$ to be chosen later,
\begin{align*}
I_M\leq & 16 e^{-a(M/16)^{\alpha/\gamma}}
\int_{\rd\times\rd} \int_{\rd} (|v|^2+|w|^2)(|v|^\gamma+|y|^\gamma)[e^{a|v|^\alpha}+e^{a|y|^\alpha}] f(dy)R(dv,dw)\\
\leq & C  e^{-a(M/16)^{\alpha/\gamma}}
\int_{\rd\times\rd} \int_{\rd} (1+|w|^2)[e^{2a|v|^\alpha}+e^{2a|y|^\alpha}] f(dy)R(dv,dw) \\
\leq&  C  e^{-a(M/16)^{\alpha/\gamma}} \int_{\rd\times\rd} \int_{\rd} \Big(1+|w|^{2+\gamma} 
+e^{\frac{4+2\gamma}{1+\gamma}a |v|^\alpha}+e^{\frac{4+2\gamma}{1+\gamma}a|y|^\alpha}\Big) 
f(dy)R(dv,dw) 
\end{align*}
by the Young inequality. Choosing $a=(1+\gamma)/(4+2\gamma)$, setting
$\kappa=a/16^{\alpha/\gamma}$ and using that 
$R \in \cH(f,g)$, we conclude that
$$
I_M \leq C  e^{-\kappa M^{\alpha/\gamma}} (1+m_{2+\gamma}(g)+\cE_\alpha(f))
$$
as desired.
\end{proof}

\section{Well-posedness}\label{secwp}

We now have all the weapons to give the

\begin{proof}[Proof of Theorem \ref{mr1}]
We fix $\gamma \in [0,1]$. If $\gamma=0$, we assume that $f_0 \in \cP_2(\rd)$ and consider the
weak solution $(f_t)_{t\geq 0}$ to \eqref{HL3D} built in Proposition \ref{ex}, which indeed satisfies
all the properties of the statement. If $\gamma\in (0,1]$, we assume that $f_0 \in \cP_2(\rd)$ 
satisfies $\cE_{\alpha}(f_0)<\infty$ for some $\alpha \in (\gamma,2)$ and consider the
weak solution $(f_t)_{t\geq 0}$ to \eqref{HL3D} built in Proposition \ref{ex}, which also satisfies
all the properties of the statement. In particular,
$\sup_{t\geq 0} \cE_{\alpha}(f_t)<\infty$ depends only on $\gamma,\alpha$ and on 
(an upperbound of) $\cE_{\alpha}(f_0)$.
We consider another weak solution $(g_t)_{t\geq 0}$ to \eqref{HL3D}, only assumed
to lie in $L^\infty_{loc}([0,\infty),\cP_{2+\gamma}(\rd))$.

\vip

{\it Step 1.} We consider $V_0 \sim f_0$ and $W_0\sim g_0$ such that $\E[|V_0-W_0|^2]=\cW_2^2(f_0,g_0)$
and a 3D Brownian motion $(B_t)_{t\geq 0}$, independent of $(V_0,W_0)$.
We consider the pathwise unique solution to 
$$
V_t=V_0+\intot b(f_s,V_s)ds + \intot \raa(f_s,V_s)dB_s,
$$
see Proposition \ref{f3}, and we know that $V_t \sim f_t$ for all $t\geq 0$.
Next, we recall that the matrix $U_\e$ was introduced in Lemma \ref{tictactoc}-(ii)
and is locally Lipschitz continuous, so that it is not difficult to verify,
as in the proof of Proposition \ref{f3}-(i), that the SDE
(with stochastic parameter $(V_s)_{s\geq 0}$)
\begin{align}\label{relou}
W_t^\e = W_0 + \intot b(g_s,W_s^\e)ds 
+ \intot \raa(g_s,W_s^\e) U_\e(a(f_s,V_s),a(g_s,W_s^\e)) dB_s
\end{align}
has a pathwise unique local solution.
But the matrix $U_\e(a(f_s,V_s),a(g_s,W_s^\e))$ being a.s. orthogonal
for all $s\geq 0$, the process $B_t^\e=\intot U_\e(a(f_s,V_s),a(g_s,W_s^\e) dB_s$
is a 3D Brownian motion. We conclude that the SDE \eqref{relou} is, in law,
equivalent to to the SDE $W_t = W_0 + \intot b(g_s,W_s)ds 
+ \intot \raa(g_s,W_s)dB_s$. We know from Proposition \ref{f3}-(i) that such a process
does not explode in finite time, so that the unique solution to \eqref{relou} is global,
and from Proposition \ref{f3}-(ii) that $W_t^\e \sim g_t$ for all $t\geq 0$.
Consequently, we have $\cW_2^2(f_t,g_t) \leq \E[|V_t-W_t^\e|^2]$ for all values of $t\geq 0$ and 
$\e\in(0,1)$.

\vip

{\it Step 2.} We set $u^\e_t=\E[|V_t-W_t^\e|^2]$. 
Computing $|V_t-W_t^\e|^2$ with the It\^o formula, taking expectations and differentiating the 
obtained expression with respect to time, we find
\begin{align*}
\frac{d}{dt}u_t^\e= \E\Big[\|\raa(f_t,V_t)-\raa(g_t,W_t^\e)&
U_\e(a(f_t,V_t),a(g_t,W_t^\e))\|^2 \\
&+ 2(V_t-W_t^\e)\cdot(b(f_t,V_t)-b(g_t,W^\e_t))\Big].
\end{align*}
Denoting by $R^\e_t \in \cP_2(\rd\times\rd)$ the law of $(V_t,W^\e_t)$ and recalling the notation of
Proposition \ref{f4}, we realize that 
$(d/dt) u^\e_t = \Gamma_\e(R^\e_t)$.

\vip

Assume first that $\gamma=0$. Then Proposition \ref{f4}
tells us that 
$(d/dt) u_t^\e \leq C \sqrt \e (1+m_2(f_t+g_t))^{1/2}$.
Recalling that $f,g \in L^\infty_{loc}([0,\infty),\cP_2(\rd))$, that 
$\cW_2^2(f_t,g_t) \leq u^\e_t$ for all $t\geq 0$ and all $\e\in (0,1)$, and that
$\E[|V_0-W_0|^2]=\cW_2^2(f_0,g_0)$ by construction, we easily deduce that
$\cW_2^2(f_t,g_t)\leq \cW_2^2(f_0,g_0)$ for all $t\geq 0$. Of course, the uniqueness of the weak solution
starting from $f_0$ follows.

\vip

When $\gamma \in (0,1]$, we work on $[0,T]$ for some fixed $T>0$. By Proposition \ref{f4}, 
for all $M>0$,
$$
\frac{d}{dt} u^\e_t\leq C \sqrt \e (1+m_{2+\gamma}(f_t+g_t))^{1/2}
+ M u^\e_t + C (1+m_{2+\gamma}(g_t)+\cE_{\alpha}(f_t))e^{-\kappa M^{\alpha/\gamma}}.
$$
For the rest of the step, we call $C_T$ a constant, allowed to vary from line to line,
depending only  $T,\alpha,\gamma$ 
and on (some upperbounds) of $\sup_{[0,T]} m_{2+\gamma}(g_t)$ and $\cE_\alpha(f_0)$.
We thus have
$$
\frac{d}{dt} u^\e_t\leq C_T \sqrt \e+ M u^\e_t + C_T e^{-\kappa M^{\alpha/\gamma}},
$$
whence $\sup_{[0,T]} u^\e_t \leq [u_0^\e +C_T \sqrt \e+ C_T e^{-\kappa M^{\alpha/\gamma}} ]e^{MT}$. Recalling
that $u_0^\e=\cW_2^2(f_0,g_0)$ and that $\cW_2^2(f_t,g_t) \leq u^\e_t$, we may let $\e\to 0$ and find that
$$
\sup_{[0,T]} \cW_2^2(f_t,g_t) \leq [\cW_2^2(f_0,g_0)+ C_T e^{-\kappa M^{\alpha/\gamma}} ]e^{MT}.
$$
We now choose $M=[\kappa^{-1}\log(1+1/\cW_2^2(f_0,g_0))]^{\gamma/\alpha}$, which is designed to satisfy
$e^{-\kappa M^{\alpha/\gamma}}=\cW_2^2(f_0,g_0)/(1+\cW_2^2(f_0,g_0))\leq \cW_2^2(f_0,g_0)$ and we end with
$$
\sup_{[0,T]} \cW_2^2(f_t,g_t) \leq C_T 
\cW_2^2(f_0,g_0)\exp(T(\kappa^{-1}\log(1+1/\cW_2^2(f_0,g_0))^{\gamma/\alpha}).
$$
We easily conclude, since $\alpha>\gamma$, 
that for any $\eta \in (0,1)$, $\sup_{[0,T]} \cW_2(f_t,g_t) \leq C_{\eta,T} (\cW_2(f_0,g_0))^{1-\eta}$,
the constant $C_{\eta,T}$ depending only on $\eta,T,\alpha$ and on (some upperbounds) of  
$\sup_{[0,T]} m_{2+\gamma}(g_t)$ and
$\cE_\alpha(f_0)$.
The uniqueness of the weak solution $(f_t)_{t\geq 0}$ starting from $f_0$ clearly follows.
\end{proof}

\section{Moments of the particle system}\label{secmps}

The goal of this section is to study the moments of the particle system.
The following result uses the fact that the particle system
a.s. conserves kinetic energy. Sznitman \cite{s} and Mischler-Mouhot \cite{mm}
have handled similar computations for the Boltzmann equation for hard spheres.

\begin{prop}\label{mps}
Fix $\gamma\in [0,1]$, $N\geq 2$, consider an exchangeable  
$(\rd)^N$-valued random variable $(V^{i,N}_0)_{i=1,\dots,N}$ and the corresponding 
unique solution $(V^{i,N}_t)_{t\geq 0}$ to \eqref{ps}.
Then for all $p > 2$,
$\sup_{[0,\infty)}\E[|V^{1,N}_t|^p]\leq C_p (\E[|V^{1,N}_0|^{p+\gamma}])^{p/(p+\gamma)}$, the constant
$C_p$ depending only on $p$ and $\gamma$.
\end{prop}

\begin{proof}
We fix $N\geq 2$ and write $V^{i}_t=V^{i,N}_t$ for simplicity.
We recall from Proposition \ref{pswp} that a.s., for all $t\geq 0$,
$E_t^N:=N^{-1} \sum_1^N|V^{i}_t|^2=E_0^N$. We fix $p>2$ and we set $u^p_t=\E[|V^1_t|^p]$.

\vip

{\it Step 1.}
Starting from \eqref{ps} and applying the It\^o formula
with $\phi(v)=|v|^p$, for which $\partial_k \phi(v)=p|v|^{p-2}v_k$ and 
$\partial_{kl}\phi(v)=p[\indiq_{\{k=l}|v|^{p-2}+(p-2)v_kv_l |v|^{p-4}]$,
we get
\begin{align*}
\frac d{dt}u^p_t =& \frac p {2N} \sum_{j=1}^N \E\Big[2 |V^1_t|^{p-2}V^1_t\cdot b(V^1_t-V^j_t)\\
&\hskip3cm +|V^1_t|^{p-2} \Tr\; a(V^1_t-V^j_t) + (p-2)|V^1_t|^{p-4}a(V^1_t-V^j_t)V^1_t\cdot V^1_t
\Big].
\end{align*}
Recalling \eqref{sev}, using exchangeability
and that everything vanishes when $j=1$, we find
\begin{align*}
\frac d{dt}u^p_t 
=&\frac {p(N-1)} {2N} \E\Big[|V^1_t-V^2_t|^\gamma\Big(-2|V^1_t|^p +2|V^1_t|^{p-2}|V^2_t|^2 \\
&\hskip3cm+(p-2)|V^1_t|^{p-4}
(|V^1_t|^2|V^2_t|^2 -(V^1_t\cdot V^2_t)^2 \Big) \Big]\\
\leq & \frac {p(N-1)} N \E\Big[|V^1_t-V^2_t|^\gamma\Big(-|V^1_t|^p+ \frac p 2 |V^1_t|^{p-2}|V^2_t|^2\Big)\Big].
\end{align*}

{\it Step 2.} When $\gamma=0$, we thus have $(d/dt)u^p_t \leq -(p/2)u^p_t + (p^2/2)\E[|V^1_t|^{p-2}|V^2_t|^2]$.
We then use exchangeability to write
$$
\E[|V^1_t|^{p-2}|V^2_t|^2]=\E\Big[|V^1_t|^{p-2} \frac1{N-1}\sum_2^N |V^i_t|^2\Big]
\leq 2 \E[|V^1_t|^{p-2} E^N_t]\leq 2 (u_t^p)^{(p-2)/p}\E[(E^N_t)^{p/2}]^{2/p}.
$$
But $\E[(E^N_t)^{p/2}]=\E[(E^N_0)^{p/2}]\leq \E[|V^1_0|^{p}]$ by Jensen's inequality and exchangeability.
We end with
$$
\frac d{dt} u_t^{p} \leq -\frac p 2 u^{p}_t + p^2  \E[|V^1_0|^{p}]^{2/p}[u^{p}_t]^{1-2/p}.
$$
We classically conclude that $\sup_{[0,\infty)} u^p_t \leq \max\{u^p_0,\E[|V^1_0|^{p}](2p)^{p/2}\}
=\E[|V^1_0|^{p}](2p)^{p/2}$.

\vip

{\it Step 3.} We suppose next that $\gamma\in(0,1]$. We know from Desvillettes and Villani \cite[Lemma 1]{dv}
that there are some constants $\kappa_p>0$ and $C_p$ such that for all $x,y \geq 0$, 
$$
-x^p-y^p+ \frac p2 x^2 y^{p-2} + \frac p 2 y^{2}x^{p-2} \leq -\kappa_p x^p + C_p(x y^{p-1} + y x^{p-1}).
$$
We deduce, using exchangeability, that
\begin{align*}
\frac d{dt}u^p_t
\leq & \frac {p(N-1)} {2N} \E\Big[|V^1_t-V^2_t|^\gamma\Big(-|V^1_t|^p-|V^2_t|^p + \frac p 2 |V^1_t|^{2}|V^2_t|^{p-2}
+ \frac p 2 |V^2_t|^{2}|V^1_t|^{p-2}\Big)\Big]\\
\leq & \frac {p(N-1)} {2N}\E\Big[|V^1_t-V^2_t|^\gamma\Big(-\kappa_p |V^1_t|^p +C_p |V^1_t||V^2_t|^{p-1}
+C_p|V^2_t||V^1_t|^{p-1}\Big)\Big]\\
\leq & \E\Big[|V^1_t-V^2_t|^\gamma\Big(-\frac{p\kappa_p}4 |V^1_t|^p +2 p C_p |V^1_t||V^2_t|^{p-1} \Big)\Big].
\end{align*}
Changing now the values of $\kappa_p>0$ and $C_p$ (which still depend only on $p$)
and using that $|v-w|^\gamma \geq \big||v|-|w|\big|^\gamma \geq |v|^\gamma - |w|^\gamma$
and $|v-w|^\gamma\leq |v|^\gamma+|w|^\gamma$, we easily find
\begin{align*}
\frac d{dt}u^p_t
\leq & - \kappa_p \E[|V^1_t|^{p+\gamma}] + 
C_p\E[|V^1_t|^p |V^2_t|^\gamma + |V^1_t|^{1+\gamma}|V^2_t|^{p-1}+|V^1_t||V^2_t|^{p-1+\gamma}]\\
\leq & - \kappa_p \E[|V^1_t|^{p+\gamma}] + 
C_p\E[|V^1_t|^p |V^2_t|^\gamma + |V^1_t|^\gamma|V^2_t|^{p}].
\end{align*}
But 
\begin{align*}
\E[|V^1_t|^p |V^2_t|^\gamma]=\E\Big[|V^1_t|^p \frac1{N-1}\sum_2^N |V^i_t|^\gamma\Big]
\leq 2 \E\Big[|V^1_t|^p  \frac 1 N \sum_1^N |V^i_t|^\gamma\Big]\leq 2 \E\Big[|V^1_t|^p (E^N_t)^{\gamma/2}\Big].
\end{align*}
By H\"older's inequality and since $E^N_t=E^N_0$, we deduce that
$$
\E[|V^1_t|^p |V^2_t|^\gamma]\leq 2\E[|V^1_t|^{p+\gamma}]^{p/(p+\gamma)}\E[(E^N_0)^{(p+\gamma)/2}]^{\gamma/(p+\gamma)}.
$$
A last application of H\"older's inequality shows that $\E[(E^N_0)^{(p+\gamma)/2}]\leq \E[|V^{1,N}_0|^{p+\gamma}]$,
whence finally
\begin{align}\label{jab}
\frac d{dt}u^p_t
\leq & - \kappa_p \E[|V^1_t|^{p+\gamma}] + C_p\E[|V^{1,N}_0|^{p+\gamma}]^{\gamma/(p+\gamma)}
\E[|V^1_t|^{p+\gamma}]^{p/(p+\gamma)} \nonumber \\
\leq & - \frac{\kappa_p}2 \E[|V^1_t|^{p+\gamma}] + C_p \E[|V^{1,N}_0|^{p+\gamma}] \nonumber \\
\leq & - \frac{\kappa_p}2 (u_t^p)^{(p+\gamma)/p} + C_p \E[|V^{1,N}_0|^{p+\gamma}],
\end{align}
the value of $C_p$ depending only on $p,\gamma$ and changing from line to line.
For the second inequality, we used that for $\kappa,a,x\geq 0$, $-\kappa x + a x^{p/(p+\gamma)}
\leq -(\kappa/2)x + (2/\kappa)^{p/\gamma} a^{(p+\gamma)/\gamma}$: it suffices to separate the cases
$\kappa x \geq  2 a x^{p/(p+\gamma)}$ and $\kappa x \leq  2a x^{p/(p+\gamma)}$.
We classically deduce from \eqref{jab} that 
$\sup_{[0,\infty)} u^p_t \leq \max\{u^p_0,(2C_p \E[|V^{1,N}_0|^{p+\gamma}]/\kappa_p)^{p/(p+\gamma)}\}$. Since 
$u^p_0=\E[|V^{1,N}_0|^{p}]\leq \E[|V^{1,N}_0|^{p+\gamma}]^{p/(p+\gamma)}$, the proof is complete.
\end{proof}

\section{Propagation of chaos}\label{secps}
The goal of this section is to check Theorem \ref{mr2}, except the time uniformity
in the Maxwell case.

\subsection{The setting}\label{set}
We consider, 
in the whole section, $\gamma\in [0,1]$ fixed and $f_0 \in \cP_2(\rd)$.
If $\gamma \in (0,1]$, we assume moreover that $\cE_{\alpha}(f_0)<\infty$ for some $\alpha \in(\gamma,2)$.
We denote by $(f_t)_{t\geq 0}$ the unique solution to \eqref{HL3D}, as well as,
for each $N\geq 2$, the unique solution $(V_t^{i,N})_{i=1,\dots,N,t\geq 0}$ to \eqref{ps}
starting from a given exchangeable $(\rd)^N$-valued $(V_0^{i,N})_{i=1,\dots,N}$. We suppose that 
$M_p:=m_p(f_0)+\sup_N \E[|V^{1,N}_0|^p] <\infty$ and we conclude from Theorem \ref{mr1} and 
Proposition \ref{mps} that for all $p\geq 2$,
$\sup_{[0,\infty)}m_p(f_t) + \sup_{N\geq 2} \sup_{[0,\infty)}\E[|V^{1,N}_t|^p]<\infty$
and depends only on $\gamma,p$ and on some (upperbound of) $M_{p+\gamma}$.
If $\gamma \in (0,1]$, we know that $\sup_{t\geq 0}\cE_{\alpha}(f_t)<\infty$.
If finally $H(f_0)<\infty$, 
then  $H(f_t)\leq H(f_0)$ for all $t\geq 0$, whence, by 
Lemma \ref{ellip1}, 
\begin{equation}\label{elli}
\sup_{t\geq 0} \sup_{v \in \rd} ||(a(f_t,v))^{-1}|| <\infty
\end{equation}
and depends only on $\gamma$ and on (upperbounds of) $m_2(f_0)$ and $H(f_0)$.

\vip

In the whole section, we write $C$ for a constant depending only on $\gamma$, $\alpha$, on (upperbounds
of) $\{M_p,p\geq 2\}$ and additionally on (an upperbound of) $\cE_{\alpha}(f_0)$ if $\gamma\in (0,1]$.
It is also allowed to depend on (an upperbound of) $H(f_0)$ when the latter is supposed to be finished.
Finally, any other dependence will be indicated in subscript.

\vip

We fix $N\geq 2$ for the whole section, we recall that 
$\mu^N_t=N^{-1}\sum_1^N \delta_{V^{i,N}_t}$ and we put $\e_N=N^{-1}$.
By \cite[Proposition A.1]{fh}, we can find $(W_0^{i,N})_{i=1,\dots,N} \sim f_0^{\otimes N}$ such that

\vip

(a) $\{(V^{i,N}_0,W^{i,N}_0), i=1,\dots ,N\}$ is exchangeable, 

\vip

(b) $\cW_2^2(N^{-1}\sum_1^N \delta_{V^{i,N}_0},
N^{-1}\sum_1^N \delta_{W^{i,N}_0})=N^{-1}\sum_1^N |V^{i,N}_0-W^{i,N}_0|^2$ a.s., 

\vip

(c) denoting by $F_N$ the law of $(V_0^{i,N})_{i=1,\dots,N}$, 
$\cW_2^2(F_N,f_0^{\otimes N})=\E[\sum_1^N|V^{i,N}_0-W^{i,N}_0|^2 ]$.

\subsection{A first coupling}

We first rewrite suitably the particle system.

\begin{lem}\label{reps}
For each $i=1,\dots,N$, the process
$$
\beta^{i,N}_t=\frac 1 {\sqrt N}\sum_{j=1}^N\intot [\raa(\mu^N_s,V^{i,N}_s)]^{-1}
\sigma(V^{i,N}_s-V^{j,N}_s)dB^{ij}_s
$$
is a 3D Brownian motion.
Furthermore, for all $i=1,\dots,N$, all $t\geq 0$, 
$$
V^{i,N}_t=V^{i}_0+\intot b(\mu^N_s,V^{i,N}_s)ds + \intot \raa(\mu^N_s,V^{i,N}_s) 
d\beta^{i,N}_s.
$$
\end{lem}

\begin{rk}\label{14}
Observe that 
$\raa(\mu^N_s,V^{i,N}_s)=[N^{-1}\sum_{j=1}^N a(V^{i,N}_s-V^{j,N}_s)]^{1/2}$
with $a(x)=[\sigma(x)]^2$.
If $\raa(\mu^N_s,V^{i,N}_s)$ is not invertible, we use Lemma \ref{matop} to define
$[\raa(\mu^N_s,V^{i,N}_s)]^{-1}\sigma(V^{i,N}_s-V^{j,N}_s)$.
We thus always have 

\vip

(i) for all $j=1,\dots, N$, 
$\raa(\mu^N_s,V^{i,N}_s)[\raa(\mu^N_s,V^{i,N}_s)]^{-1}\sigma(V^{i,N}_s-V^{j,N}_s)=\sigma(V^{i,N}_s-V^{j,N}_s)$;

\vip

(ii) $N^{-1}\sum_{j=1}^N ([\raa(\mu^N_s,V^{i,N}_s)]^{-1}\sigma(V^{i,N}_s-V^{j,N}_s))
([\raa(\mu^N_s,V^{i,N}_s)]^{-1}\sigma(V^{i,N}_s-V^{j,N}_s))^*=\Id$.
\end{rk}

\begin{lem}\label{matop}
For $A_1,\dots,A_N \in S_3^+$ and $M=N^{-1}\sum_1^NA_j^2$, we can find some matrices $B_1,\dots,B_N$
such that (a) $M^{1/2}B_j=A_j$ for all $j=1,\dots,N$ and (b) $N^{-1}\sum_1^N B_jB_j^* = \Id$.
We write $B_j=M^{-1/2}A_j$, even in the case where $M$ is not invertible.
\end{lem}

\begin{proof}
If $M$ is invertible, it suffices to set $B_j=M^{-1/2}A_j$. If $M=0$, the choice $B_j=\Id$ is suitable.
Assume now that $M$ has exactly two non-trivial eigenvalues $\lambda_1,\lambda_2>0$ (the last case where
$M$ has exactly one non-trivial eigenvalue is treated similarly). Consider an orthonormal basis
$e_1,e_2,e_3$ of eigenvectors, that is, $Me_1=\lambda_1 e_1$, $Me_2=\lambda_2 e_2$ and $Me_3=0$
(so that $M=\lambda_1e_1e_1^*+\lambda_2e_2e_2^*)$
and observe that $A_je_3=0$ for all $j$. It then suffices to set 
$B_j=(\lambda_1^{-1/2}e_1e_1^*+\lambda_2^{-1/2}e_2e_2^*)A_j + e_3 e_3^*$.
\end{proof}

We can now give the

\begin{proof}[Proof of Lemma \ref{reps}]
For $i$ fixed, the Brownian motions $(B^{ij})_{j \ne i}$ are independent. Hence
the (matrix) bracket of the 3D martingale $(\beta^{i,N}_t)_{t\geq 0}$
is given by (recall that $\sigma(0)=0$)
$$
\frac 1 N \sum_{j=1}^N \intot 
\Big([\raa(\mu^N_s,V^{i,N}_s)]^{-1} \sigma(V^{i,N}_s-V^{j,N}_s)\Big)
\Big([\raa(\mu^N_s,V^{i,N}_s)]^{-1} \sigma(V^{i,N}_s-V^{j,N}_s)\Big)^*  ds
=\Id t,
$$
which implies that $(\beta^{i,N}_t)_{t\geq 0}$ is a Brownian motion. We used Remark \ref{14}-(ii).
Rewriting \eqref{ps} as in the statement is straightforward, using that
$\raa(\mu^N_s,V^{i,N}_s)[\raa(\mu^N_s,V^{i,N}_s)]^{-1} 
\sigma(V^{i,N}_s-V^{j,N}_s)=\sigma(V^{i,N}_s-V^{j,N}_s) $ by Remark \ref{14}-(i).
\end{proof}

We next introduce a (non-independent) family of Landau processes. Recall that
the matrix $U$ was introduced in \eqref{dfU},  that $U_\e$ was defined in Lemma \ref{tictactoc}-(ii). Denote $\e_N=N^{-1}$.

\begin{lem}\label{treschiant}
The system of equations (for $i=1,\dots,N$)
$$
W^{i,N}_t=W^{i,N}_0+\intot b(f_s,W^{i,N}_s)ds + \intot \raa(f_s,W^{i,N}_s)U_{\e_N}(a(\mu^{N}_s,V^{i,N}_s),
a(\nu^{N}_s,W^{i,N}_s)) d\beta^{i,N}_s,
$$
with $\nu^{N}_t=N^{-1}\sum_1^N \delta_{W^{i,N}_t}$,
has a pathwise unique solution. Furthermore, $W^{1,N}_t \sim f_t$ for all $t\geq 0$ and the family 
$\{(V^{i,N}_t,W^{i,N}_t)_{t\geq 0}, i=1, \dots , N\}$ is exchangeable.
\end{lem}

\begin{proof}
As usual, the existence of a pathwise unique local solution follows from the fact that
the coefficients are locally Lipschitz continuous (which follows from Lemmas \ref{f2} and 
Lemma \ref{tictactoc}-(iii)). But for each $i$, the matrix $U_{\e_N}(a(\mu^{N}_s,V^{i,N}_s),
a(\nu^{N}_s,W^{i,N}_s))$ being a.s. orthogonal
for all $s\geq 0$, the process $\intot U_{\e_N}(a(\mu^{N}_s,V^{i,N}_s),
a(\nu^{N}_s,W^{i,N}_s)) d\beta^{i,N}_s$ is a 3D Brownian motion. 
Consequently, the SDE satisfied by $W^{i,N}$ is
equivalent (in law) to the SDE $W_t = V_0 + \intot b(f_s,W_s)ds 
+ \intot \raa(f_s,W_s)dB_s$ (with $V_0\sim f_0$). We know from Proposition \ref{f3}-(i) that such a process
does not explode in finite time, so that the unique solution is global,
and from Proposition \ref{f3}-(ii) that $W_t^{i,N} \sim f_t$ for all $t\geq 0$.
Exchangeability is obvious, using that it holds true at time $0$
(see point (a) at the end of Subsection \ref{set}).
\end{proof}

\subsection{A second coupling}
Unfortunately, the processes $(W^{i,N}_t)_{t\geq 0}$ are not independent, so we have to show that they are
{\it almost} independent in some sense.

\begin{lem}\label{lf}
For all $K=1,\dots,N$, we can find an i.i.d. family of processes 
$(Z^{i,N,K}_t)_{i=1,\dots,K,t\geq 0}$ such that
$Z^{i,N,K}_t\sim f_t$ for all $t\geq 0$, all $i=1,\dots,K$
and such that for all $\eta \in (0,1)$, all $T>0$,
\begin{equation}\label{lef}
\sup_{i=1,\dots,K} \sup_{[0,T]} \E[|W^{i,N}_t-Z^{i,N,K}_t|^2] \leq C_{\eta,T} 
\frac K {N^{1-\eta}}.
\end{equation}
Moreover, the constant $C_{\eta,T}$ is of the form $C_\eta T$ if $\gamma=0$.
\end{lem}

\begin{proof}
Let $K \in \{1,\dots,N\}$ and $\eta\in (0,1/2)$ be fixed for the whole proof. We also put $\delta=(K/N)^2 >0$.
For simplicity, we write $V^i_s=V^{i,N}_s$, $W^{i}_s=W^{i,N}_s$ and $Z^{i}_s=Z^{i,N,K}_s$.

\vip

{\it Step 1.}
We recall that the Brownian motions $(B^{ij})_{1\leq i < j\leq N}$ are independent, that $B^{ij}=-B^{ji}$
and we introduce a new family $(\tB^{ij})_{1\leq i,j \leq N}$ of independent Brownian motions
(also independent of everything else). We recall that the Brownian motions $\beta^{i,N}_t$ were defined
in Lemma \ref{reps} and we introduce, for $i=1,\dots,K$, 
$$
\tbeta^{i,N}_t=\frac 1 {\sqrt N}\sum_{j=1}^N\intot [\raa(\mu^N_s,V^{i,N}_s)]^{-1}
\sigma(V^{i,N}_s-V^{j,N}_s)d[\indiq_{\{j\leq K\}}\tB^{ij}_s+\indiq_{\{j> K\}}B^{ij}_s].
$$
One easily checks, using Remark \ref{14}-(ii), 
that the continuous 3D martingales $\tbeta^{1,N},\dots,\tbeta^{K,N}$ satisfy
$\langle\tbeta^{i,N},\tbeta^{j,N}\rangle_t= \Id t \indiq_{\{i=j\}}$, so that they are independent 3D Brownian motions.
We next claim that the system of equations (for $i=1,\dots,K$)
$$
Z^{i}_t=W^{i}_0+\intot b(f_s,Z^{i}_s)ds + \intot \raa(f_s,Z^{i}_s)
X^{i}_sU^{i}_s d \tbeta^{i,N}_s,
$$
where we have set $U^{i}_s=U_{\e_N}(a(\mu^{N}_s,V^{i}_s),
a(\nu^{N}_s,W^{i}_s))$ and $X^{i}_s=U_\delta(a(f_s,W^{i}_s),a(f_s,Z^{i}_s))$ for simplicity, 
has a pathwise unique solution (with the same arguments as usual, see the proof of Lemma \ref{treschiant})
and that for each $i=1,\dots,K$, $Z^{i}_t \sim f_t$ for all $t\geq 0$.
Furthermore, the Brownian motions $\intot X^{i}_sU^{i}_s d \tbeta^{i,N}_s$ being independent
(as orthogonal martingales with deterministic brackets), as well as the initial conditions
$W^i_0$, the pathwise uniqueness stated in Lemma \ref{f3}-(i) implies that the processes
$(Z^{i}_t)_{t\geq 0}$, for $i=1,\dots,K$, are independent.
It only remains to prove \eqref{lef} and, by exchangeability, it suffices to study $\E[|W^1_t-Z^1_t|^2]$.

\vip

{\it Step 2.} Here we verify that, denoting by $R^N_t$ the law of $(W^1_t,Z^1_t)$,
of which the two marginals equal $f_t$ ans using the notation of Proposition \ref{f4}, we have 
$$
\frac d {dt}\E[|W^1_t-Z^1_t|^2] \leq C_\eta K N^{\eta-1}+\Gamma_\delta(R^N_t).
$$
Recalling the equations satisfied by $W^1$ (see Lemma \ref{treschiant}) and $Z^1$, as well
as the expressions of $\beta^{1,N}$ (see Lemma \ref{reps}) and $\tbeta^{1,N}$, we see that
\begin{align*}
W^1_t-&Z^1_t=\intot [b(f_s,W^1_s)-b(f_s,Z^1_s)]ds\\
&+ \frac 1 {\sqrt{N}} \sum_{j=K+1}^N \intot\Big[\raa(f_s,W^1_s)-\raa(f_s,Z^1_s)X^1_s \Big]
U^1_s[\raa(\mu^N_s,V^{1,N}_s)]^{-1}\sigma(V^{1,N}_s-V^{j,N}_s)dB^{1j}_s\\
&+ \frac 1 {\sqrt{N}} \sum_{j=1}^K \intot \raa(f_s,W^1_s)U^1_s[\raa(\mu^N_s,V^{1,N}_s)]^{-1}
\sigma(V^{1,N}_s-V^{j,N}_s)dB^{1j}_s\\
&-  \frac 1 {\sqrt{N}} \sum_{j=1}^K \intot \raa(f_s,Z^1_s)X^1_sU^1_s[\raa(\mu^N_s,V^{1,N}_s)]^{-1}
\sigma(V^{1,N}_s-V^{j,N}_s)d\tB^{1j}_s.
\end{align*}
All the Brownian motions appearing in this formula are independent.
By the It\^o formula, we find
$(d/dt)\E[|W^1_t-Z^1_t|^2]=\E[I_1+I_2+I_3+I_4]$, with
\begin{align*}
I_1=&2(W^1_t-Z^1_t)\cdot [b(f_t,W^1_t)-b(f_t,Z^1_t)],\\
I_2=&\frac1N\sum_{j=K+1}^N \Big\|[\raa(f_t,W^1_t)-\raa(f_t,Z^1_t)X^1_t]U^1_t[\raa(\mu^N_t,V^{1,N}_t)]^{-1} 
\sigma(V^{1,N}_t-V^{j,N}_t) \Big\|^2,\\
I_3=&\frac1N\sum_{j=1}^K \Big\|[\raa(f_t,W^1_t)U^1_t[\raa(\mu^N_t,V^{1,N}_t)]^{-1} 
\sigma(V^{1,N}_t-V^{j,N}_t) \Big\|^2,\\
I_4=&\frac1N\sum_{j=1}^K \Big\|[\raa(f_t,Z^1_t)U^1_tX^1_t[\raa(\mu^N_t,V^{1,N}_t)]^{-1} 
\sigma(V^{1,N}_t-V^{j,N}_t) \Big\|^2.
\end{align*}

Using that $N^{-1}\sum_{j=1}^N[\sigma(V^{1,N}_t\!-V^{j,N}_t)]^2
=[\raa(\mu^N_t,V^{1,N}_t)]^2$ and that $\|A\|^2=\Tr\;AA^*$, we find
\begin{align*}
I_2\leq& \frac1N\sum_{j=1}^N \Big\|[\raa(f_t,W^1_t)-\raa(f_t,Z^1_t)X^1_t]U^1_t[\raa(\mu^N_t,V^{1,N}_t)]^{-1} 
\sigma(V^{1,N}_t-V^{j,N}_t) \Big\|^2\\
=&\Big\| [ \raa(f_t,W^1_t)-\raa(f_t,Z^1_t)X^1_t ]U^1_t\Big\|^2\\
=&\Big\| \raa(f_t,W^1_t)-\raa(f_t,Z^1_t)X^1_t\Big\|^2
\end{align*}
because $U^1_t$ is a.s. an orthogonal matrix.
Recalling the notation of Proposition \ref{f4} and that $X^1_t=U_\delta(a(f_t,W^{i}_t),a(f_t,Z^{i}_t))$, 
we conclude that $\E[I_1+I_2]\leq \Gamma_\delta(R^N_t)$. 

\vip

By exchangeability, we have, for $q> 1$ and $q'=q/(q-1)$, by H\"older's inequality,
\begin{align*}
\E[I_3]\leq & \frac K N \E\Big[\Big\|\raa(f_t,W^1_t)\Big\|^2\Big\|[\raa(\mu^N_t,V^{1,N}_t)]^{-1} 
\sigma(V^{1,N}_t-V^{2,N}_t)\Big\|^2 \Big]\\
\leq & \frac K N \E\Big[\Big\|\raa(f_t,W^1_t)\Big\|^{2q}\Big]^{1/q}
\E\Big[\Big\|[\raa(\mu^N_t,V^{1,N}_t)]^{-1} \sigma(V^{1,N}_t-V^{2,N}_t)\Big\|^{2q'} \Big]^{1/q'}.
\end{align*}
By Lemma \ref{f2}, since $W^1_t \sim f_s$ and since $\sup_{[0,\infty)} m_{(2+\gamma)q}(f_s) <\infty$
(see Subsection \ref{set}), we have 
$\E[\|\raa(f_t,W^1_t)\|^{2q}]^{1/q}\leq C_q$. Next, we have a.s.
$$
\|[\raa(\mu^N_t,V^{1,N}_t)]^{-1} \sigma(V^{1,N}_t-V^{2,N}_t)\|^2 \leq \sum_{j=1}^N 
\|[\raa(\mu^N_t,V^{1,N}_t)]^{-1} \sigma(V^{1,N}_t-V^{j,N}_t)\|^2=N
$$
by Remark \ref{14}-(ii) and, by exchangeability, 
$$
\E[\|[\raa(\mu^N_t,V^{1,N}_t)]^{-1} \sigma(V^{1,N}_t-V^{2,N}_t)\|^2]=N^{-1}\E[\sum_{j=1}^N 
\|[\raa(\mu^N_t,V^{1,N}_t)]^{-1} \sigma(V^{1,N}_t-V^{j,N}_t)\|^2]=1.
$$
Consequently,
\begin{align*}
\E[I_3]\leq & C_q \frac K N \E\Big[\Big\|[\raa(\mu^N_t,V^{1,N}_t)]^{-1} \sigma(V^{1,N}_t-V^{2,N}_t)\Big\|^{2} 
N^{2(q'-1)} \Big]^{1/q'} = C_q K N^{2(q'-1)/q'-1}.
\end{align*}
Choosing $q= 2/\eta$, we find that $2(q'-1)/q'=\eta$, whence $\E[I_3]\leq C_\eta K N^{\eta-1}$.
Finally, $I_4$ is treated exactly as $I_3$ and this ends the step.

\vip

{\it Step 3.} If $\gamma=0$, by Proposition \ref{f4}-(i), 
$\Gamma_\delta(R^N_t)\leq C\sqrt{\delta}(1+m_2(f_t))^{1/2}
\leq C \sqrt{\delta}$,
so that we end with $(d/dt)\E[|W^1_t-Z^1_t|^2]\leq C\sqrt{\delta} + C_\eta  K N^{\eta-1}\leq C_\eta K N^{\eta-1}$
(because $\delta=(K/N)^2$).
Since $W^1_0=Z^1_0$, we conclude that $\E[|W^1_t-Z^1_t|^2]\leq  C_\eta  K N^{\eta-1} T$ as desired.

\vip

{\it Step 4.} Assume next that $\gamma \in (0,1]$. We then have
$\sup_{[0,\infty)} [m_{2+\gamma}(f_t)+\cE_\alpha(f_t)]<\infty$, see Subsection \ref{set}.
We thus infer from Proposition \ref{f4}-(ii) that for all $M>0$,
$$
\Gamma_\delta(R^N_t) \leq C \sqrt \delta + M\int_{\rd\times\rd}|v-w|^2R^N_t(dv,dw)+ C e^{-\kappa M^{\alpha/\gamma}}.
$$
But $\sqrt \delta = K N^{-1}$ and $\int_{\rd\times\rd}|v-w|^2R^N_t(dv,dw)=\E[|W^1_t-Z^1_t|^2]$, 
so that we have proved that
$(d/dt)\E[|W^1_t-Z^1_t|^2]\leq C_\eta  K N^{\eta-1} + M\E[|W^1_t-Z^1_t|^2] + 
C e^{-\kappa M^{\alpha/\gamma}}$ and thus
$$
\sup_{[0,T]}\E[|W^1_t-Z^1_t|^2] \leq [C_\eta T K N^{\eta-1}+C T e^{-\kappa M^{\alpha/\gamma}}]e^{MT}.
$$
Choosing $M=[\kappa^{-1}\log(1+K^{-1}N^{1-\eta})]^{\gamma/\alpha}$, for which
$e^{-\kappa M^{\alpha/\gamma}}=1/(1+K^{-1}N^{1-\eta})\leq KN^{\eta-1}$,
$$
\sup_{[0,T]}\E[|W^1_t-Z^1_t|^2] \leq C_\eta T
K N^{\eta-1}\exp\Big(T[\kappa^{-1}\log(1+K^{-1}N^{1-\eta})]^{\gamma/\alpha}\Big).
$$
Since $\gamma<\alpha$, this is easily bounded by $C_{\eta,T} (K N^{\eta-1})^{1-\eta}\leq C_{\eta,T} K N^{-(1-\eta)^2}
\leq C_{\eta,T} K N^{2\eta-1}$.
\end{proof}

A first consequence of the previous Lemma is the following quantitative law of large numbers.

\begin{lem}\label{LLN}
Consider a function $\varphi : \rd \mapsto \rr$ satisfying
$|\varphi(x)-\varphi(y)|\leq C |x-y|(1+|x|^q+|y|^q)$ for some $q\geq 2$. As usual,
we set $\varphi(\mu,x)=\intrd \varphi(x-y)\mu(dy)$ for any probability measure $\mu$ on $\rd$.
Then for all $T>0$, all $\eta \in (0,1/2)$,
$$
\sup_{[0,T]}\E\big[|\varphi(\nu^{N}_t,W^{1,N}_t)-\varphi(f_t,W^{1,N}_t)|^2\big] 
\leq C_{\eta,T,\varphi} N^{\eta - 1/2}.
$$
Moreover, the constant $C_{\eta,T,\varphi}$ is of the form $C_{\eta,\varphi} \sqrt T$ if $\gamma=0$.
\end{lem}

\begin{proof}
Using exchangeability, we write 
\begin{align*}
\E\big[|\varphi(\nu^{N}_t,W^{1,N}_t)-\varphi(f_t,W^{1,N}_t)|^2\big] 
= & \frac 1 {N^2}\E\Big[\Big(\sum_{i=1}^N [\varphi(W^{1,N}_t-W^{i,N}_t)- \varphi(f_t,W^{1,N}_t)] \Big)^2\Big]\\
=&\frac1{N^2}\Big( I_1+ 2(N-1)I_2+ (N-1)I_3+(N-1)(N-2)I_4\Big),
\end{align*}
with (we develop the squared sum and separate the cases (a) $i=j=1$, (b) $i=1$ and $j\ne 1$ or $i\ne 1$ and $j=1$, 
(c) $i=j \ne 1$, (d) $i\ne j$, $i\ne 1$, $j\ne 1$)
\begin{align*}
I_1=& \E[(\varphi(0)-\varphi(f_t,W^{1,N}_t) )^2 ],\\
I_2=&\E[(\varphi(0)-\varphi(f_t,W^{1,N}_t) )
(\varphi(W^{1,N}_t-W^{2,N}_t)- \varphi(f_t,W^{1,N}_t)) ]\\
I_3=&\E[(\varphi(W^{1,N}_t-W^{2,N}_t)- \varphi(f_t,W^{1,N}_t))^2 ]\\
I_4=& \E[(\varphi(W^{1,N}_t-W^{2,N}_t)- \varphi(f_t,W^{1,N}_t))
(\varphi(W^{1,N}_t-W^{3,N}_t)- \varphi(f_t,W^{1,N}_t))].
\end{align*}
Using only that $\varphi$ has at most polynomial growth, that $W^{1,N}_t\sim W^{2,N}_t\sim f_t$
and that all the moments of $f_t$ are uniformly (in time) bounded, we easily verify that
$I_1+I_2+I_3 \leq C_{\varphi}$, whence $N^{-2}(I_1+ 2(N-1)I_2+ (N-1)I_3) \leq C_\varphi N^{-1}$.
We next use Lemma \ref{lf} with $K=3$ to write
$I_4 \leq J_1+J_2+J_3$, with
\begin{align*}
J_1=&\E[(\varphi(Z^{1,N,3}_t-Z^{2,N,3}_t)- \varphi(f_t,Z^{1,N,3}_t))
(\varphi(Z^{1,N,3}_t-Z^{3,N,3}_t)- \varphi(f_t,Z^{1,N,3}_t))],\\
J_2=&\E[(|\varphi(W^{1,N}_t-W^{2,N}_t)-\varphi(Z^{1,N,3}_t-Z^{2,N,3}_t)|
+|\varphi(f_t,W^{1,N}_t)-\varphi(f_t,Z^{1,N,3}_t)  |) \\
&\hskip1cm \times |\varphi(Z^{1,N,3}_t-Z^{3,N,3}_t)- \varphi(f_t,Z^{1,N,3}_t)|],\\
J_3=&\E[|\varphi(W^{1,N}_t-W^{2,N}_t)- \varphi(f_t,W^{1,N}_t)|  \\
&\hskip1cm \times(|\varphi(W^{1,N}_t-W^{3,N}_t)-\varphi(Z^{1,N,3}_t-Z^{3,N,3}_t)|
+|\varphi(f_t,W^{1,N}_t)-\varphi(f_t,Z^{1,N,3}_t)  |)].
\end{align*}
But $J_1=0$ because $ Z^{1,N,3}_t,Z^{2,N,3}_t,Z^{3,N,3}_t$ are independent and $f_t$-distributed:
it suffices to first take the conditional expectation knowing $Z^{1,N,3}_t$ and to observe 
that $\E[\varphi(Z^{1,N,3}_t-Z^{2,N,3}_t) \vert Z^{1,N,3}_t]=\varphi(f_t,Z^{1,N,3}_t)$.
Next, using that all the variables $W^{1,N}_t,W^{2,N}_t,W^{3,N}_t,Z^{1,N,3}_t,Z^{2,N,3}_t,Z^{3,N,3}_t$
are $f_t$-distributed, that  all the moments of $f_t$ are uniformly bounded, that 
$\varphi$ has at most polynomial growth, the local Lipschitz property of $\varphi$,
and that $|\varphi(f_t,w)-\varphi(f_t,z)| \leq C_\varphi |w-z|(1+m_q(f_t)+|w|^q+|z|^q)$,
we easily get convinced that, by exchangeability and the Cauchy-Schwarz inequality,
$$
\E[J_2+J_3] \leq C_\varphi \E[|W^{1,N}_t-Z^{1,N,3}_t|^2]^{1/2}.
$$
This is bounded by $C_{\eta,T,\varphi} N^{\eta-1/2}$ by Lemma \ref{lf} with $K=3$, and the
constant  $C_{\eta,T,\varphi}$ is of the form $C_{\eta,\varphi}\sqrt T$ in the case where $\gamma=0$. 
\end{proof}

\subsection{Computation of the error}
We now handle the main computation of the proof.

\begin{lem}\label{big}
For all $T>0$, all $\eta \in (0,1/4)$, all $t\in [0,T]$,
$$
\E[|V^{1,N}_t-W^{1,N}_t|^2] \leq 
\begin{cases} 
C_{\eta} (1+T)^{5/2} \big(\E[|V^{1,N}_0-W^{1,N}_0|^2]+N^{\eta-1/4}\big)  & \hbox{if $\gamma=0$},\\
C_{\eta} (1+T)^{5/2} \big(\E[|V^{1,N}_0-W^{1,N}_0|^2]+N^{\eta-1/2}\big)  & \hbox{if $\gamma=0$ and $H(f_0)<\infty$},\\
C_{\eta,T} \big(\E[|V^{1,N}_0-W^{1,N}_0|^2]+N^{-1/4}\big)^{1-\eta} & \hbox{if $\gamma \in (0,1]$,}\\
C_{\eta,T} \big(\E[|V^{1,N}_0-W^{1,N}_0|^2]+N^{-1/2}\big)^{1-\eta} & \hbox{if $\gamma \in (0,1]$ and $H(f_0)<\infty$}.
\end{cases}
$$
\end{lem}

\begin{proof} For simplicity, we write $V^{i}_t=V^{i,N}_t$, $W^{i}_t=W^{i,N}_t$ and
$U^i_t=U_{\e_N}(a(\mu^N_t,V^{i,N}_t),a(\nu^N_t,W^{i,N}_t))$. Also, we set
$u^N_t=\E[|V^{1,N}_t-W^{1,N}_t|^2]$.
For each $t\geq 0$, we define $\zeta^{N}_t=N^{-1}\sum_1^N \delta_{(V^{i,N}_t,W^{i,N}_t)}$, which
a.s. belongs to $\cH(\mu^N_t,\nu^N_t)$. We fix $\eta \in (0,1/9)$ and $T>0$ 
and we work on $[0,T]$.

\vip

{\it Step 1.} Recalling the equations satisfied by $V^1$ (see Lemma \ref{reps}) and $W^1$ 
(see Lemma \ref{treschiant}), 
the It\^o formula leads us to
\begin{align*}
\frac d{dt} u^N_t=&\E\Big[2(V^1_t-W^1_t)\cdot(b(\mu^N_t,V^1_t)-b(f_t,W^1_t))
+ \|\raa(\mu^N_t,V^1_t)-\raa(f_t,W^1_t) U^1_t \|^2 \Big]\\
=&\E\Big[2(V^1_t-W^1_t)\cdot(b(\mu^N_t,V^1_t)-b(\nu^N_t,W^1_t))
+ \|\raa(\mu^N_t,V^1_t)-\raa(\nu^N_t,W^1_t) U^1_t \|^2 \Big]\\
&+\E\Big[2(V^1_t-W^1_t)\cdot(b(\nu^N_t,W^1_t)-b(f_t,W^1_t)) \Big]\\
&+\E\Big[ \|(\raa(\nu^N_t,W^1_t)-\raa(f_t,W^1_t))U^1_t \|^2 \Big]\\
&+2\E\Big[\lps\raa(\mu^N_t,V^1_t)-\raa(\nu^N_t,W^1_t) U^1_t,
(\raa(\nu^N_t,W^1_t)-\raa(f_t,W^1_t))U^1_t \rps\Big]
\end{align*}
Using now that $U^1_t$ is an orthogonal matrix and the Cauchy-Schwarz inequality, we find that
$$
\frac d{dt} u^N_t \leq \E[I^N_t] + 2 \sqrt{u^N_t\E[J^N_t] }+\E[K^N_t]+2\E\Big[\sqrt{L^N_t K^N_t}\Big],
$$
where
\begin{align*}
I^N_t=&2(V^1_t-W^1_t)\cdot(b(\mu^N_t,V^1_t)-b(\nu^N_t,W^1_t))
+ \|\raa(\mu^N_t,V^1_t)-\raa(\nu^N_t,W^1_t) U^1_t \|^2,\\
J^N_t=&|b(\nu^N_t,W^1_t)-b(f_t,W^1_t)|^2,\\
K^N_t=&\|\raa(\nu^N_t,W^1_t)-\raa(f_t,W^1_t) \|^2, \\
L^N_t=&\|\raa(\mu^N_t,V^1_t)-\raa(\nu^N_t,W^1_t) U^1_t\|^2.
\end{align*}

{\it Step 2.} We first prove that $\E[I^N_t]=\E[\Gamma_{\e_N}(\zeta^N_t)]$.
Using exchangeability,
\begin{align*}
\E[I^N_t]=&\E\Big[\frac 1N \sum_{i=1}^N [2(V^i_t-W^i_t)\cdot(b(\mu^N_t,V^i_t)-b(\nu^N_t,W^i_t))
+ \|\raa(\mu^N_t,V^i_t)-\raa(\nu^N_t,W^i_t) U^i_t \|^2]. 
\end{align*}
It then suffices to recall that $\zeta^{N}_t=N^{-1}\sum_1^N \delta_{(V^{i,N}_t,W^{i,N}_t)}$,
of which the marginals are $\mu^N_t$ and $\nu^N_t$, that
$U^i_t=U_{\e_N}(a(\mu^N_t,V^{i}_t),a(\nu^N_t,W^{i}_t))$
and the notation of Proposition \ref{f4}.

\vip

{\it Step 3.} Using Lemma \ref{LLN} and the Lipschitz property of $b$ checked in Lemma \ref{f1},
we immediately get that $\E[J^N_t] \leq C_{\eta,T} N^{\eta-1/2}$, with moreover $C_{\eta,T}=C_\eta T^{1/2}$
if $\gamma=0$.

\vip

{\it Step 4.} Here we verify that 

\vip

(i) we always have $\E[K^N_t] \leq C_{\eta,T} N^{\eta-1/4}$, 
with moreover $C_{\eta,T}=C_\eta  T^{1/4}$
if $\gamma=0$; 

\vip

(ii) if $H(f_0)<\infty$, then $\E[K^N_t] \leq C_{\eta,T} N^{\eta-1/2}$, with moreover 
$C_{\eta,T}=C_\eta  T^{1/2}$
if $\gamma=0$.
 
\vip

For (i), we use the first inequality of Lemma \ref{math} to write
$K^N_t=\| \raa(\nu^N_t,W^1_t)-\raa(f_t,W^1_t)\|^2 \leq C \|a(\nu^N_t,W^1_t)-a(f_t,W^1_t)\|$.
We then apply Lemma \ref{LLN}, which is licit thanks to the Lipschitz property of $a$ checked in Lemma \ref{f1},
to get $\E[K^N_t] \leq C_{\eta,T} N^{\eta-1/4}$, with $C_{\eta,T}=C_\eta  T^{1/4}$
if $\gamma=0$.

\vip

For point (ii), we use the second inequality of Lemma \ref{math} and then the ellipticity 
estimate \eqref{elli} to write 
$K^N_t\leq C \|[a(f_t,W^1_t)]^{-1}\| \|a(\nu^N_t,W^1_t)-a(f_t,W^1_t)\|^2\leq \|a(\nu^N_t,W^1_t)-a(f_t,W^1_t)\|^2$.
Again, Lemma \ref{LLN}
implies that $\E[K^N_t] \leq C_{\eta,T} N^{\eta-1/2}$, with moreover $C_{\eta,T}=C_\eta T^{1/2}$
if $\gamma=0$.

\vip

{\it Step 5.} We now check that $\E[\sqrt{K^N_tL^N_t}]\leq \E[I^N_t]+ \E[K^N_t] + C_\eta 
\sqrt{u^N_t} \E[K^N_t]^{(1-\eta)/2}$.
We first observe that by Lemma \ref{f1},
\begin{align*}
L^N_t =& I^N_t - 2 (V^1_t-W^1_t)\cdot(b(\mu^N_t,V^1_t)-b(\nu^N_t,W^1_t)) \\
=& I^N_t -2 (V^1_t-W^1_t)\cdot \frac 1N \sum_{i=1}^N (b(V^1_t-V^i_t)-b(W^1_t-W^i_t))\\
\leq & I^N_t + C |V^1_t-W^1_t| \frac 1 N  \sum_{i=1}^N (|V^1_t-W^1_t|+|V^i_t-W^i_t|)
(1+|V^1_t|+|W^1_t|+|V^i_t|+|W^i_t|)^\gamma\\
\leq & I^N_t + C |V^1_t-W^1_t|\Big(\frac 1 N  \sum_{i=1}^N  (|V^1_t-W^1_t|^2+|V^i_t-W^i_t|^2)\Big)^{1/2}
\sqrt{H^N_t}\\
\leq & I^N_t + C  M^N_t \sqrt{H^N_t},
\end{align*}
where we have set 
$$
H^N_t=\frac 1N\sum_{i=1}^N (1+|V^1_t|+|W^1_t|+|V^i_t|+|W^i_t|)^{2\gamma} \quad \hbox{and}\quad
M^N_t=|V^1_t-W^1_t|^2+\frac 1N\sum_{i=1}^N |V^i_t-W^i_t|^2.$$
Since $\sqrt{x(y+z)}\leq \sqrt{xy}+\sqrt{xz}\leq x+y+\sqrt{xz}$, we conclude that 
\begin{align*}
\E\big[\sqrt{K^N_tL^N_t}\big]\leq& \E[I^N_t]+ \E[K^N_t] + C\E[(K^N_t)^{1/2} (M^N_t)^{1/2} (H^N_t)^{1/4}]\\
=& \E[I^N_t]+ \E[K^N_t] + C\E[(M^N_t)^{1/2} (K^N_t)^{(1-\eta)/2} ((K^N_t)^{\eta/2}(H^N_t)^{1/4})]\\
\leq & \E[I^N_t]+ \E[K^N_t] +C \E[M^N_t]^{1/2}\E[K^N_t]^{(1-\eta)/2}\E[((K^N_t)^{\eta/2}(H^N_t)^{1/4})^{2/\eta}]^{\eta/2},
\end{align*}
where we used the triple H\"older inequality with $p=2$, $q=2/(1-\eta)$ and $r=2/\eta$ for the last inequality.
But it holds that $\E[M^N_t]=2u^N_t$ by exchangeability. 
To complete the step, it only remains to prove that $\E[((K^N_t)^{\eta/2}(H^N_t)^{1/4})^{2/\eta}]\leq C_\eta$.
But $K^N_t \leq 2 \| \raa(\nu^N_t,W^1_t)\|^2+2\|\raa(f_t,W^1_t)\|^2 \leq C(m_{2+\gamma}(f_t+\nu^N_t)+|W^1_t|^2)$
by Lemma \ref{f2}
and $H^N_t\leq  1+|V^1_t|^{2\gamma} +|W^1_t|^{2\gamma}+m_{2\gamma}(\mu^N_t+\nu^N_t)$.
Observing that $\E[(m_{2\gamma}(\mu^N_t+\nu^N_t))^p ]\leq \E[|V^1_t|^{2\gamma p}+|W^1_t|^{2\gamma p}]$
by H\"older's inequality (if $p\geq 1$), that $W^{1,N}_t\sim f_t$ and recalling that
$\sup_{[0,\infty)}m_p(f_t) + \sup_{N\geq 2} \sup_{[0,\infty)}\E[|V^{1,N}_t|^p]<\infty$
for all $p>2$ (see Subsection \ref{set}), we conclude that $K^N_t$ and $H^N_t$
have uniformly bounded moments of all orders, so that finally, 
$\E[((K^N_t)^{\eta/2}(H^N_t)^{1/4})^{2/\eta}]\leq C_\eta$.

\vip

{\it Step 6.} From Steps 1 and 5,
$(d/dt)u^N_t\leq 3\E[I^N_t]+3\E[K^N_t]+C_\eta\sqrt{u^N_t(\E[J^N_t] + \E[K^N_t]^{1-\eta})}$. Using Steps 3 and 4,
we see that $\E[J^N_t] + \E[K^N_t]+ \E[K^N_t]^{1-\eta} \leq \delta_{\eta,T,N}$, where 
\vip
(i) $\delta_{\eta,T,N}= C_{\eta,T} (N^{\eta-1/4})^{1-\eta}\leq C_{\eta,T} N^{2\eta-1/4}$ in general
(with $C_{\eta,T}=C_\eta (1+T)^{1/2}$ if $\gamma=0$);
\vip
(ii) $\delta_{\eta,T,N}= C_{\eta,T} (N^{\eta-1/2})^{1-\eta}\leq N^{2\eta-1/2}$ if $H(f_0)<\infty$
(with $C_{\eta,T}=C_\eta (1+T)^{1/2}$ if $\gamma=0$). 
\vip
Using finally Step 2, we end with 
$$
\frac d{dt} u^N_t\leq 3\E[\Gamma_{\e_N}(\zeta^N_t)]+\delta_{\eta,T,N} + C \sqrt{u^N_t \delta_{\eta,T,N}}.
$$

\vip

{\it Step 7.} Assume that $\gamma=0$. By Proposition \ref{f4}-(i) (recall that $\e_N=N^{-1}$),
$$
\E[\Gamma_{\e_N}(\zeta^N_t)]\leq C \sqrt{\e_N} \E[1+m_2(\mu^N_t+\nu^N_t)]^{1/2}
=C \sqrt{\e_N} \E[|V^{1,N}_t|^2+|W^{1,N}_t|^2]^{1/2} \leq C\sqrt{\e_N}\leq C \delta_{\eta,T,N}.
$$
Thus 
$(d/dt)u^N_t\leq C \delta_{\eta,T,N} +  C \sqrt{ u^N_t \delta_{\eta,T,N}}\leq C \sqrt{\delta_{\eta,T,N}^2+ 
u^N_t \delta_{\eta,T,N}}$. Integrating this differential inequality, 
we deduce that $\sup_{[0,T]}u^N_t \leq C(1+T)^2(u_0^N +\delta_{\eta,T,N})$,
from which the conclusion follows.

\vip

{\it Step 8.} Assume next that $\gamma \in (0,1]$ and recall that $\alpha>\gamma$. 
By Proposition \ref{f4}-(ii), for all $M>0$,
\begin{align*}
\Gamma_{\e_N}(\zeta^N_t)\leq& C \sqrt{\e_N}(1+m_{2+\gamma}(\mu^N_t+\nu^N_t))^{1/2}
+ M \int_{\rd\times\rd} |v-w|^2\zeta^N_t(dv,dw) \\
&+ C e^{-\kappa M^{\gamma/\alpha}}(m_{2+\gamma}(\mu^N_t)+\cE_\alpha(\nu^N_t)).
\end{align*}
We have $\E[m_{2+\gamma}(\mu^N_t+\nu^N_t)]=\E[|V^1_t|^{\gamma+2}]+m_{2+\gamma}(f_t)\leq C$,
see Subsection \ref{set}. Also,
it holds  that $\E[\cE_\alpha(\nu^N_t)]=\cE_\alpha(f_t)\leq C$. Finally, 
$\E[\int_{\rd\times\rd} |v-w|^2\zeta^N_t(dv,dw)]=u^N_t$. All in all, we have checked that 
$\E[\Gamma_{\e_N}(\zeta^N_t)]\leq C \sqrt{\e_N} + C e^{-\kappa M^{\gamma/\alpha}} + M u^N_t$. 
Recalling Step 6, using that $\sqrt{\e_N}\leq C \delta_{\eta,T,N}$ and that $\sqrt{xy}\leq x+y$, we conclude that
$$
\frac{d}{dt} u^N_t \leq C \delta_{\eta,T,N}+(M+C)u^N_t+ C e^{-\kappa M^{\gamma/\alpha}},
$$
whence $u^N_t \leq  [u_0^N+C T \delta_{\eta,T,N}+ CT e^{- \kappa M^{\gamma/\alpha}}]e^{(M+C)T}$.
As usual, we make the choice $M=[\kappa^{-1}\log(1+1/[u^N_0+\delta_{\eta,T,N}])]^{\gamma/\alpha}$, for which
$e^{-\kappa M^{\alpha/\gamma}}\leq u^N_0+\delta_{\eta,T,N}$, and this leads us to
$$
u^N_t \leq  C_T [u_0^N+\delta_{\eta,T,N}]\exp(T[\kappa^{-1}\log(1+1/[u_0^N+\delta_{\eta,T,N}])]^{\gamma/\alpha}+CT)
\leq C_{\eta,T} [u_0^N+\delta_{\eta,T,N}]^{1-\eta}, 
$$
because $\gamma<\alpha$.
We conclude that $u^N_t\leq C_{\eta,T} (u^N_0+N^{2\eta-1/4})^{1-\eta}\leq C_{\eta,T} (u^N_0+N^{-1/4})^{1-9\eta}$
in general and $u^N_t\leq C_{\eta,T} (u^N_0+N^{2\eta-1/2})^{1-\eta}\leq C_{\eta,T} (u^N_0+N^{-1/2})^{1-5\eta}$
if $H(f_0)<\infty$. 
\end{proof}

\subsection{A quantified law of large numbers for non independent variables}
Here we check the following result, to be applied soon to the family $W^{i,N}_t$.

\begin{lem}\label{untiers}
Let $N\geq 2$, $\mu \in \cP_5(\rd)$, $\eta \in (0,1)$ and $\kappa>0$.
Consider an exchangeable family $W_1,\dots,W_N$ of $\rd$-valued random variables such that for all $K=1,\dots,N$,
there are some i.i.d. $\mu$-distributed random variables $Z^K_1,\dots,Z^K_K$ such that
$\max_{i=1,\dots,K}\E[|W_i-Z^{K}_i|^2]\leq \kappa K N^{\eta-1}$. Then 
$$
\E\Big[\cW_2^2\Big( \frac 1 N \sum_1^N \delta_{W_i},\mu \Big)\Big]
\leq \frac{C (1+m_5(\mu)^{2/5}+\kappa)}{N^{(1-\eta)/3}}, 
$$
where $C$ is a universal constant.
\end{lem}

\begin{proof} We divide the proof into four steps.

\vip

{\it Step 1.} We recall the well-known fact that for $f,f',g,g'\in\cP_2(\rd)$ and $\lambda\in(0,1)$, it holds that
$\cW_2^2(\lambda f + (1-\lambda) g,\lambda f' + (1-\lambda) g')\leq \lambda \cW_2^2(f,g)+(1-\lambda)\cW_2^2(f',g')$.
Indeed, consider $X\sim f$ and $Y\sim g$ such that $\E[|X-Y|^2]=\cW_2^2(f,g)$,  $X'\sim f'$ and $Y'\sim g'$ 
such that $\E[|X'-Y'|^2]=\cW_2^2(f',g')$, and $U\sim$ Bernoulli$(\lambda)$, with $(X,Y),(X',Y'),U$ independent.
Then $Z:=UX+(1-U)Y\sim \lambda f + (1-\lambda) g$, $Z':=UX'+(1-U)Y'\sim \lambda f' + (1-\lambda) g'$, and
one easily verifies that $\E[|Z-Z'|^2]=\lambda\E[|X-Y|^2]+(1-\lambda)\E[|X'-Y'|^2]
=\lambda \cW_2^2(f,g)+(1-\lambda)\cW_2^2(f',g')$.

\vip

{\it Step 2.}
For $K\in\{1,\dots, N \}$, we set $\mu_K=K^{-1}\sum_{i=1}^K \delta_{W_i}$. We prove in this step that 
$$
\E[\cW_2^2(\mu_N,\mu)] \leq \E[\cW_2^2(\mu_K,\mu)]+  \frac {(6m_2(\mu)+4\kappa)K} N.
$$
To this end, we set $R=\lfloor N/K\rfloor$ and we assume that $RK<N$, the other case being easier 
(no need to introduce $\nu^N_{R+1}$).
We introduce, for $k=1,\dots,R$, $\nu^N_k=K^{-1}\sum_{i=(k-1)K+1}^{kK}\delta_{W_{i}}$,
as well as $\nu^N_{R+1}=(N-RK)^{-1}\sum_{i=RK+1}^N \delta_{W_i}$. 
We then write $\mu_N= K N^{-1} \sum_{k=1}^R \nu^N_k + (N-RK)N^{-1} \nu^N_{R+1}$
and we use Step 1 to obtain $\cW_2^2(\mu_N,\mu)\leq K N^{-1} \sum_{k=1}^R\cW_2^2(\nu^N_k,\mu) + (N-RK)N^{-1}
\cW_2^2(\nu^N_{R+1},\mu)$. By exchangeability, we thus find
$$
\E[\cW_2^2(\mu_N,\mu)] \leq \frac {RK} N \E[\cW_2^2(\nu^N_1,\mu)]+ \frac{N-RK}{N}\E[\cW_2^2(\nu^N_{R+1},\mu)].
$$
The conclusion follows, because $RK\leq N$, because $\nu^N_1=\mu_K$, because 
$N-RK\leq K$ and because $\E[\cW_2^2(\nu^N_{R+1},\mu)]\leq 2m_2(\mu)+2\E[|W_1|^2]
\leq 2m_2(\mu)+4\E[|Z_1^1|^2]+4\E[|W_1-Z_1^1|^2]\leq 6m_2(\mu)+4\kappa$.

\vip

{\it Step 3.} We then introduce $\zeta_K=K^{-1}\sum_{i=1}^K \delta_{Z^K_i}$. 
Since the $Z^K_i$'s are i.i.d. and $\mu$-distributed, we know from 
\cite[Theorem 1]{fgui} (with $p=2$, $d=3$ and $q=5$) that for all $K=1,\dots,N$,
$\E[\cW_2^2(\zeta_K,\mu)]\leq C (m_5(\mu))^{2/5} K^{-1/2}$. Next, we have
$\E[\cW_2^2(\mu_K,\zeta_K)]\leq K^{-1}\sum_1^K \E[|W_i-Z^K_i|^2]\leq \kappa K N^{\eta-1}$. Consequently,
$\E[\cW_2^2(\mu_K,\mu)]\leq C (m_5(\mu))^{2/5} K^{-1/2} + 2 \kappa K N^{\eta-1}$.

\vip

{\it Step 4.} Gathering Steps 2 and 3, we find that for all $K\in\{1,\dots,N\}$,
$$
\E[\cW_2^2(\mu_N,\mu)] \leq \frac{C (m_5(\mu))^{2/5}}{\sqrt K} + \frac{2\kappa K}{N^{1-\eta}}
+  \frac {(6m_2(\mu)+4\kappa)K} N \leq C(1+ (m_5(\mu))^{2/5} + \kappa ) 
\Big[\frac 1 {\sqrt{K}}+\frac{K}{N^{1-\eta}} \Big].
$$
Choosing $K=\lfloor N^{2(1-\eta)/3}\rfloor$ completes the proof.
\end{proof}

\subsection{Conclusion}
 
We now have all the weapons to prove Theorem \ref{mr2}, except the time uniformity in the Maxwell
case. We start with the case of hard potentials.

\begin{proof}[Proof of Theorem \ref{mr2}-(ii)]
We thus assume that $\gamma \in (0,1]$ and we fix $T>0$ and $\eta \in (0,1)$.
We recall that $\mu^N_t=N^{-1}\sum_1^N \delta_{V^{i,N}_t}$ and $\nu^N_t=N^{-1}\sum_1^N \delta_{W^{i,N}_t}$ and we write
$\cW_2^2(\mu^N_t,f_t)\leq 2 \cW_2^2(\mu^N_t,\nu^N_t)+2\cW_2^2(\nu^N_t,f_t)$.
Lemma \ref{lf} (together with the fact that $\sup_{t\geq 0} m_5(f_t)<\infty$) allows us to apply 
Lemma \ref{untiers} to obtain $\sup_{[0,T]} \E[\cW_2^2(\nu^N_t,f_t)] \leq C_{\eta,T} N^{(\eta-1)/3}
\leq C_{\eta,T} N^{\eta-1/3}$.
Next, we write $\E[\cW_2^2(\mu^N_t,\nu^N_t)]\leq N^{-1}\sum_1^N\E[|V^{i,N}_t-W^{i,N}_t|^2]=
\E[|V^{1,N}_t-W^{1,N}_t|^2]$. We conclude from Lemma \ref{big} that $\sup_{[0,T]}\E[\cW_2^2(\mu^N_t,\nu^N_t)]$
is controlled by $C_{\eta,T}(\E[|V^{1,N}_0-W^{1,N}_0|^2]+N^{-1/4})^{1-\eta}$ 
in general and by $C_{\eta,T}(\E[|V^{1,N}_0-W^{1,N}_0|^2]+N^{-1/2})^{1-\eta}$ if $H(f_0)<\infty$.
By points (a) and (b) stated at the end of Subsection \ref{set},
$\E[|V^{1,N}_0-W^{1,N}_0|^2]=\E[N^{-1} \sum_1^N |V^{i,N}_0-W^{i,N}_0|^2]=\E[\cW_2^2(\mu^N_0,\nu^N_0)]
\leq 2 \E[\cW_2^2(\mu^N_0,f_0)]+2\E[\cW_2^2(\nu^N_0,f_0)]$. And $\E[\cW_2^2(\nu^N_0,f_0)]\leq C N^{-1/2}$
by \cite[Theorem 1]{fgui}, because $(W_0^{i,N})_{i=1,\dots,N} \sim f_0^{\otimes N }$ and $m_5(f_0)<\infty$.
All in all, we can bound $\sup_{[0,T]}\E[\cW_2^2(\mu^N_t,f_t)]$ by 
$C_{\eta,T}(\E[\cW_2^2(\mu^N_0,f_0)]+N^{-1/4})^{1-\eta}$, and even
by $C_{\eta,T}(\E[\cW_2^2(\mu^N_0,f_0)]+N^{-1/3})^{1-\eta}$ if $H(f_0)<\infty$.
\end{proof}

Proceeding similarly, we find the following weak version of Theorem \ref{mr2}-(i).

\begin{thm}\label{maxweak}
Assume that $\gamma=0$. Fix $f_0 \in \cP_2(\rd)$ and consider the 
corresponding unique weak solution $(f_t)_{t\geq 0}$ to \eqref{HL3D}.
For each $N\geq 2$, consider an exchangeable $(\rd)^N$-valued random variable $(V^{i,N}_0)_{i=1,\dots,N}$
and the corresponding unique solution $(V_t^{i,N})_{i=1,\dots,N,t\geq 0}$ to \eqref{ps}. Set
$\mu^N_t=N^{-1}\sum_1^N \delta_{V^{i,N}_t}$. Assume that for all $p\geq 2$,
$M_p:=m_p(f_0)+\sup_{N\geq 2}\E[|V^{1,N}_0|^p]<\infty$.
For all $\eta \in (0,1/4)$, there is a constant $C_\eta$ depending
only on $\eta$, on (some upperbounds of) $\{M_p,p\geq 2\}$ and on (some upperbound of) $H(f_0)$ 
when it is finite such that
$$
\E[\cW_2^2(\mu^N_t,f_t)]\leq \begin{cases}
C_\eta(1+t)^{5/2}(\E[\cW_2^2(\mu^N_0,f_0)]+N^{\eta-1/4}) & \hbox{in general,}\\
C_\eta(1+t)^{5/2}(\E[\cW_2^2(\mu^N_0,f_0)]+ N^{\eta-1/3})& \hbox{if } H(f_0)<\infty.
\end{cases}
$$
\end{thm}

\section{Uniform convergence to equilibrium in the Maxwell case}\label{secmax}

We now prove, when $\gamma=0$, the uniform (in $N$) convergence to equilibrium
of the particle system, following the arguments of Rousset \cite{r}. We will easily 
deduce the time-uniformity of the propagation of chaos.

\vip

In the whole section, we assume that $\gamma=0$.
For $N\geq 2$ and for $F_N$ an exchangeable law on $(\rd)^N$, we call $\bL^N(F_N)\in \cP(C([0,\infty), (\rd)^N)$
the law of the solution $(V^{i,N}_t)_{i=1,\dots,N,t\geq 0}$ to \eqref{ps}
with $(V^{i,N}_0)_{i=1,\dots,N}\sim F_N$. We also put 
$\bL^N_t(F_N)=\cL((V^{i,N}_t)_{i=1,\dots,N}) \in \cP((\rd)^N)$ for each $t\geq 0$.
We introduce
$$
\cS_N=\Big\{(v_1,\dots,v_N)\in (\rd)^N \; : \; N^{-1}\sum_1^N v_i=0 , \; N^{-1}\sum_1^N |v_i|^2=1\Big\}.
$$

\begin{rk}\label{equi}
The uniform distribution on $S_N$ is invariant:
$\bL^N_t(\cU(S_N))=\cU(S_N)$ for all $t\geq 0$.
\end{rk}

This observation is classical and actually holds true for any value of $\gamma \in [0,1]$.
To give a precise reference, let us mention that in 
\cite[Theorem 4.2-(ii)]{c}, Carrapatoso shows that under some conditions on $F_N\in\cP(S_N)$,
$\cW_1(\bL^N_t(F_N),\cU(S_N))$ tends to $0$ as $t\to \infty$, which implies that $\cU(S_N)$ is invariant.

\begin{thm}\label{rouss}
Fix $N\geq 7$ and some exchangeable
$(V^{i,N}_0)_{i=1,\dots,N}\sim F_N \in \cP(\cS_N)$. For all $p>0$,
there is a constant $C_p$ depending only on $p$ such that if $N\geq 6+2p$,
for all $t\geq 0$,
\begin{align*}
\frac 1 N\cW_2^2(\bL^N_t(F_N),\cU(S_N)) \leq& 
\min\Big\{\frac 1 N\cW_2^2(F_N,\cU(S_N)), \frac{ C_p \E[1+|V^{1,N}_0|^{8+4p}]^{1/2}}{(1+t)^{p}} \Big\}.
\end{align*}
\end{thm}

Although we slightly clarify some points and although the coupling is slightly more technical
for the Landau equation, the proof closely follows \cite{r}.
In the next subsection, we recall some facts about $\cU(S_N)$.
We build a suitable coupling in Subsection \ref{sscoup} and recall Rousset's main inequality in Subsection
\ref{ssrou}. We conclude the proof of Theorem \ref{rouss} in Subsection \ref{ssconc}.
Finally, we deduce Theorem \ref{mr2}-(i) from Theorems \ref{maxweak} and \ref{rouss} in Subsection \ref{ssff}.

\subsection{The uniform law on the sphere} We will need the following facts.

\begin{lem}\label{ums}
Let $(X_1^N,\dots,X_N^N)\sim \cU(\cS_N)$. Then

\vip

(i) $\E[\cW_2^2(N^{-1}\sum_1^N\delta_{X_i^N},\cN(0,3^{-1}\Id))] \leq C N^{-1/2}$;

\vip

(ii) for all $p\geq 1$, $\E[|X_1^N|^p]\leq C_p$, where $C_p$ depends only on $p$;

\vip

(iii) if $1\leq p \leq N-4$, for $\rho_N$ the spectral radius of $M_N=N^{-1}\sum_1^N X_i^N (X_i^N)^*$,
we have $\E[(1-\rho_N)^{-p}]\leq C_p$, where $C_p$ depends only on $p$.
\end{lem}

We will use twice the following observation.

\begin{rk}\label{rkd}
For any $f,g\in\cP_2(\rd)$, $\cW_2^2(f,g)\geq [(V_f)^{1/2}-(V_g)^{1/2}]^2 + |m_f-m_g|^2$,
where $m_f=\intrd vf(dv)$ and $V_f=\intrd |v-m_f|^2f(dv)$. 
\end{rk}
Indeed, 
for any  $X\sim f$ and $Y \sim g$,
$\E[|X-Y|^2]= \E[|(X-\E[X])-(Y-\E[Y])|^2]+|\E[X-Y]|^2 \geq  V_f+V_g -2(V_fV_g)^{1/2} + |m_f-m_g|^2$.

\begin{proof}[Proof of Lemma \ref{ums}] 
Consider an i.i.d. sample $(Y_1,\dots,Y_N)$ of the $\cN(0,3^{-1}\Id)$ distribution. Define
$m_N=N^{-1}\sum_1^N Y_i$, $E_N=N^{-1}\sum_1^N|Y_i-m_N|^2$ and $X^N_i=E_N^{-1/2}(Y_i-m_N)$. 
Then it is classical (see e.g. \cite[Proof of Lemma 4.3]{r}) that $(X_1^N,\dots,X_N^N)\sim \cU(\cS_N)$.

\vip

To prove (i), we set $\mu_N=N^{-1}\sum_1^N \delta_{X^N_i}$ and $\nu_N=N^{-1}\sum_1^N \delta_{Y_i}$.
We have $\E[\cW_2^2(\mu_N,\nu_N)] \leq N^{-1}\sum_1^N\E[|Y_i-X^N_i|^2]=N^{-1}\sum_1^N\E[|(Y_i-m_N)(1-1/\sqrt{E_N})
+m_N|^2]=\E[(1-\sqrt{E_N})^2+|m_N|^2]$. By Remark \ref{rkd} and since $m_{\nu_N}=m_N$ and $V_{\nu_N}=E_N$,
we conclude that  $\E[\cW_2^2(\mu_N,\nu_N)] \leq \E[\cW_2^2(\nu_N,\cN(0,3^{-1}\Id)]$,
so that $\E[\cW_2^2(\mu_N,\cN(0,3^{-1}\Id)] \leq 4\E[\cW_2^2(\nu_N,\cN(0,3^{-1}\Id)]$.
By \cite[Theorem 1]{fgui}, it holds that $\E[\cW_2^2(\nu_N,\cN(0,3^{-1}\Id)) ] \leq C N^{-1/2}$
and this proves (i).

\vip
Point (ii) has been checked by Carrapatoso \cite[Lemma 10]{c}.

\vip

We finally check (iii) (see \cite[Lemma 4.4]{r} for a less precise statement), assuming that $N\geq p+4\geq 5$.
The empirical covariance matrix $A_N=\sum_1^N (Y_i-m_N) (Y_i-m_N)^*$ classically 
(see Anderson \cite[Section 7]{an}) follows a 
Wishart$(3,N-1)$-distribution,
and $M_N=A_N/ \Tr\; A_N$. The eigenvalues $0\leq L^N_1\leq L^N_2\leq L^N_3$ of $A_N$ are known to have the
density (see Anderson \cite[Theorem 13.3.2]{an}, this uses that $3\leq N-1$) 
$$
g_N(\ell_1,\ell_2,\ell_3)=\kappa_N^{-1} (\ell_1\ell_2\ell_3)^{(N-5)/2}[(\ell_3-\ell_2)(\ell_3-\ell_1)(\ell_2-\ell_1)]
e^{-(\ell_1+\ell_2+\ell_3)/2}\indiq_{\{0<\ell_1<\ell_2<\ell_3\}},
$$
where $\kappa_N=\pi^{-9/2}2^{3(N-1)/2}\Gamma((N-1)/2)\Gamma((N-2)/2)\Gamma((N-3)/2)\Gamma(3/2)\Gamma(1)\Gamma(1/2)$.
But $1-\rho_N=(L^N_1+L^N_2)/(L^N_1+L^N_2+L^N_3)\geq 2(L^N_1L^N_2)^{1/2}/(3L^N_3)=2(L^N_1L^N_2L^N_3)^{1/2}/(3(L^N_3)^{3/2})$.
Consequently, for $p\in [1,N-4]$ (so that $3\leq N-p-1$),
\begin{align*}
\E[(1-\rho_N)^{-p}]\leq&\Big(\frac32\Big)^p \intrd \ell_3^{3p/2}(\ell_1\ell_2\ell_3)^{-p/2}g_N(\ell_1,\ell_2,\ell_3)d\ell_1
d\ell_2d\ell_3\\
=& \Big(\frac32\Big)^p \kappa_N^{-1} \kappa_{N-p} \intrd \ell_3^{3p/2}g_{N-p}(\ell_1,\ell_2,\ell_3)d\ell_1
d\ell_2d\ell_3=\Big(\frac32\Big)^p \kappa_N^{-1} \kappa_{N-p} \E[(L^{N-p}_3)^{3p/2}].
\end{align*}
But using that $L^N_3\leq \Tr\; A_N=E_N\sim\chi^2(3N-3)$,
it is not hard to verify that $\E[(L^{N}_3)^{3p/2}]\leq C_p N^{3p/2}$. We thus end
with $\E[(1-\rho_N)^{-p}] \leq C_p N^{3p/2} \kappa_N^{-1} \kappa_{N-p}$. Using the expression of $\kappa_N$ and the 
Stirling formula, we easily conclude that $\sup_{N\geq p+4} \E[(1-\rho_N)^{-p}]<\infty$ as desired.
\end{proof}

\subsection{The coupling}\label{sscoup}
Recall that $U$ was defined in \eqref{dfU}. We need to use $U(a(x),a(y))$, which is unfortunately
not well-defined. The lemma below gives some sense to 
$A(x,y)=\sigma(y)U(a(x),a(y))$.

\begin{lem}\label{mata}
Recall that for $x\in \rd$, $\sigma(x)=|x|\Pi_{x^\perp}$ and $a(x)=|x|^2\Pi_{x\perp}$.
We can find a measurable family of $3\times 3$ matrices $(A(x,y))_{x,y\in\rd}$ 
verifying $A(-x,-y)=A(x,y)$ and
\vip
(a) $A(x,y) A^*(x,y)=a(y)$, 
\vip
(b) $\lps \sigma(x),A(x,y)\rps= |x||y|+x\cdot y$, 
\vip
(c) $(\sigma(x)-A^*(x,y))(x-y)=0$.
\end{lem}

\begin{proof}
If $x=0$, it suffices to set $A(x,y)=\sigma(y)$. Else, we consider an orthonormal basis 
$(e_1,e_2,e_3)$ satisfying $e_1=x/|x|$ and $e_3 \cdot y =0$ and we set
$A(x,y)=-(y\cdot e_2) e_1 e_2^* + (y\cdot e_1) e_2e_2^* + |y|e_3e_3^*$.
To check (a), put $y_i=y\cdot e_i$, note that $y_3=0$ and that $\Id=e_1e_1^*+e_2e_2^*+e_3e_3^*$:
direct computations show that both $a(y)=|y|^2\Id-yy^*$  and $A(x,y) A^*(x,y)$
equal $y_2^2 e_1e_1^* + y_1^2e_2e_2^*+|y|^2e_3e_3^* - y_1y_2(e_1e_2^*+e_2e_1^*)$.
For point (b), one starts with $\sigma(x)=|x|\Id-|x|^{-1}xx^*=|x|(e_2e_2^*+e_3e_3^*)$,
whence $\sigma(x)A^*(x,y)=|x|(-y_2e_2e_1^*+y_1e_2e_2^*+|y|e_3e_3^*)$ and thus
$\Tr \; \sigma(x)A^*(x,y)=|x|(y_1+|y|)=|x||y|+x\cdot y$. 
For (c), one easily finds that both $\sigma(x)(x-y)$ and $A^*(x,y)(x-y)$
equal $-y_2|x|e_2$.

\vip

Since $A(x,y)$ satisfies conditions (a)-(b)-(c) with $-x$ and $-y$, it is 
possible to handle the construction in such a way that 
$A(-x,-y)=A(x,y)$. Measurability is not an issue.
\end{proof}

We now build a suitable coupling.

\begin{lem}\label{coup}
Consider two exchangeable laws $F_N$ and $G_N$ in $\cP(S_N)$. There exists an exchangeable family
$\{(V^{i,N}_t,W^{i,N}_t)_{t\geq 0},i=1,\dots,N\}$ satisfying the following properties:

\vip

(i) $\E[\sum_1^N |V^{i,N}_0-W^{i,N}_0|^2]=\cW_2^2(F_N,G_N)$;

\vip

(ii) a.s., $N^{-1}\sum_1^N |V^{i,N}_0-W^{i,N}_0|^2=\cW_2^2(N^{-1}\sum_1^N \delta_{V^{i,N}_0},N^{-1}\sum_1^N \delta_{W^{i,N}_0})
\leq 2$;

\vip

(iii) $(V^{i,N}_t)_{i=1,\cdots,N, t\geq 0} \sim \bL^N(F_N)$ and $(W^{i,N}_t)_{i=1,\cdots,N, t\geq 0} \sim \bL^N(G_N)$;

\vip

(iv) a.s., for all $t\geq 0$,
\begin{align*}
\frac 1N \sum_{i=1}^N  & |V^{i,N}_t-W^{i,N}_t|^2=\frac 1N \sum_{i=1}^N  |V^{i,N}_0-W^{i,N}_0|^2 \\
&- \frac 2{N^2} \sum_{i,j=1}^N \intot [|V^{i,N}_s-V^{j,N}_s||W^{i,N}_s-W^{j,N}_s| - 
(V^{i,N}_s-V^{j,N}_s)\cdot(W^{i,N}_s-W^{j,N}_s) ]ds.
\end{align*}
\end{lem}

\begin{proof} The function $A(x,y)$ cannot be continuous and this 
causes some technical difficulties.
We write $V^{i}_t=V^{i,N}_t$ and $W^{i}_t=W^{i,N}_t$ for simplicity.

\vip

{\it Step 1.} By \cite[Proposition A.1]{fh},
we can find $H_N\in\cP(S_N\times S_N)$ with marginals $F_N$ and $G_N$ such that,
for $((V^{i}_0)_{i=1,\dots,N},(W^{i}_0)_{i=1,\dots,N})\sim H_N$,
the family $\{(V^{i}_0,W^{i}_0), i=1,\dots,N\}$ is exchangeable and points (i) and (ii)
hold true. Actually, the inequality in (ii) follows from the fact
that $\cW_2^2(f,g)\leq m_2(f)+m_2(g)$ (choose an independent coupling between $f$ and $g$) and that 
$m_2(N^{-1}\sum_1^N \delta_{V^{i}_0})=m_2(N^{-1}\sum_1^N \delta_{W^{i}_0})=1$ because both $F_N$ and $G_N$ are carried by
$S_N$.

\vip

{\it Step 2.} We consider $((V^{i}_0)_{i=1,\dots,N},(W^{i}_0)_{i=1,\dots,N})\sim H_N$
and some families $(B^{ij}_t)_{1\leq i < j \leq N, t\geq 0}$, $(\tB^i_t)_{i=1,\dots,N,t\geq 0}$, 
$(\hB^i_t)_{i=1,\dots,N,t\geq 0}$
of $3D$ Brownian motions, all these objects being independent. 
For $1\leq j<i\leq N$, we set $B^{ij}_t=-B^{ji}_t$.
We also put $B^{ii}_t=0$ for all $i=1,\dots,N$
and consider the system of SDEs
\begin{align}
V^{i,\e}_t=&V^{i,\e}_0 \!+\! \frac 1 N \sum_{j=1}^N \intot\!\! b(V^{i,\e}_s\!-\!V^{j,\e}_s)ds \!+\!
\frac 1 {\sqrt N} \sum_{j=1}^N \intot \!\!\sigma(V^{i,\e}_s\!-\!V^{j,\e}_s)dB^{ij}_s + \e \tB^i_t,\label{sys1}\\
W^{i,\e}_t=&W^{i,\e}_0 \!+\! \frac 1 N \sum_{j=1}^N \intot\!\! b(W^{i,\e}_s\!-\!W^{j,\e}_s)ds \!+\!
\frac 1 {\sqrt N} \sum_{j=1}^N \intot\!\! A(V^{i,\e}_s\!-\!V^{j,\e}_s, W^{i,\e}_s\!-\!W^{j,\e}_s)dB^{ij}_s + \e \hB^i_t.
\label{sys2}
\end{align}
This is a $6N$-dimensional stochastic differential equation with measurable coefficients
with at most linear growth (because $|b(x)|=2|x|$ and $\|\sigma(x)\|^2=\|A(y,x)\|^2=\Tr\; a(x)=2|x|^2$). 
Thanks to the additional noises, the diffusion matrix is strictly
uniformly elliptic. Consequently, we can apply the result of 
Krylov \cite[p 87]{k} (the coefficients are assumed to be bounded in \cite{k}, but we can use a 
standard localization procedure
or the results of Rozkosz and Slominski \cite{rs}): the system \eqref{sys1}-\eqref{sys2} 
has at least one (weak) solution.

\vip

{\it Step 3.}
We now prove that
\begin{align}\label{jab2}
\frac 1N\sum_1^N|V^{i,\e}_t&-W^{i,\e}_t|^2=\frac1N\sum_1^N|V^{i}_0-W^{i}_0|^2
+6\e^2 t + \frac{2\e }{N}\sum_{i=1}^N \intot  (V^{i,\e}_s-W^{i,\e}_s)\cdot (d\tB^i_s-d\hB^i_s)\\
&- \frac 2{N^2} \sum_{i,j=1}^N \intot [|V^{i,\e}_s-V^{j,\e}_s||W^{i,\e}_s-W^{j,\e}_s| - 
(V^{i,\e}_s-V^{j,\e}_s)\cdot(W^{i,\e}_s-W^{j,\e}_s)]ds.\nonumber
\end{align}
This follows from a direct application of the It\^o formula, together with the equalities
\begin{align*}
I^\e_s:=&\frac{1}{N^2}\sum_{i,j=1}^N \Big(2(V^{i,\e}_s-W^{i,\e}_s)\cdot(b(V^{i,\e}_s-V^{j,\e}_s)- b(W^{i,\e}_s-W^{j,\e}_s))\\
&\hskip4cm + \|\sigma(V^{i,\e}_s-V^{j,\e}_s)- A(V^{i,\e}_s-V^{j,\e}_s, W^{i,\e}_s-W^{j,\e}_s)\|^2  \Big)\\
=& - \frac 2{N^2} \sum_{i,j=1}^N  [|V^{i,\e}_s-V^{j,\e}_s||W^{i,\e}_s-W^{j,\e}_s| - 
(V^{i,\e}_s-V^{j,\e}_s)\cdot(W^{i,\e}_s-W^{j,\e}_s)]
\end{align*}
and
\begin{align*}
J^\e_t :=& \frac{2}{N\sqrt N} \sum_{i,j=1}^N \intot (V^{i,\e}_s-W^{i,\e}_s)\cdot 
\Big(\sigma(V^{i,\e}_s-V^{j,\e}_s)- A(V^{i,\e}_s-V^{j,\e}_s, W^{i,\e}_s-W^{j,\e}_s) \Big) dB^{ij}_s=0
\end{align*}
that we now check.
Using that $B^{ij}_s=-B^{ji}_s$, $\sigma(-x)=\sigma(x)$ and $A(-x,-y)=A(x,y)$, we see that 
\begin{align*}
J^\e_t =& \frac{2}{N\sqrt N} \sum_{i<j} \intot \Big((V^{i,\e}_s-V^{j,\e}_s)-(W^{i,\e}_s-W^{j,\e}_s)\Big)\\
&\hskip4cm \cdot \Big(\sigma(V^{i,\e}_s-V^{j,\e}_s)- A(V^{i,\e}_s-V^{j,\e}_s, W^{i,\e}_s-W^{j,\e}_s) \Big) dB^{ij}_s.
\end{align*}
It then suffices to use that $(v-w)^* (\sigma(v)-A(v,w))=0$
by Lemma \ref{mata}-(c) to conclude that $J^\e_t=0$. 
Using next that $b(-x)=-b(x)$, we write
\begin{align*}
I^\e_s:=&\frac{1}{N^2}\sum_{i,j=1}^N 
((V^{i,\e}_s-V^{j,\e}_s)-(W^{i,\e}_s-W^{j,\e}_s))\cdot(b(V^{i,\e}_s-V^{j,\e}_s)- b(W^{i,\e}_s-W^{j,\e}_s))\\
&\hskip4cm + \|\sigma(V^{i,\e}_s-V^{j,\e}_s)- A(V^{i,\e}_s-V^{j,\e}_s, W^{i,\e}_s-W^{j,\e}_s)\|^2  \Big).
\end{align*}
But $\|\sigma(x)-A(x,y)\|^2=\Tr\;\sigma(x)\sigma^*(x)
+\Tr\;A(x,y)A^*(x,y)-2\Tr\;\sigma(x)A^*(x,y)=2|x|^2+2|y|^2- 2(|x||y|+x\cdot y)$
because $\sigma(x)\sigma^*(x)=a(x)$, $A(x,y)A^*(x,y)=a(y)$ by Lemma \ref{mata}-(a), 
because $\Tr\; a(x)=2|x|^2$, and because $\Tr\;\sigma(x)A^*(x,y)=|x||y|+x\cdot y$ 
by Lemma \ref{mata}-(b). Also, since
$b(x)=-2x$, we have $(x-y)\cdot(b(x)-b(y))=-2|x|^2-2|y|^2+4x\cdot y$.
All in all, $(x-y)\cdot(b(x)-b(y))+\|\sigma(x)-A(x,y)\|^2=-2|x||y|+2x\cdot y$.
This completes the step.

\vip

{\it Step 4.} The coefficients $b,\sigma,A$ have at most linear growth
and the initial conditions are bounded (for $N$ fixed, since $H_N$ is carried by $S_N\times S_N$).
It is thus routine to verify that for all $p\geq 2$, all $T>0$, 
$\sup_{\e\in(0,1)}\E[\sup_{[0,T]}\sum_1^N (|V^{i,\e}_t|^p+|W^{i,\e}_t|^p)]<\infty$ and that the family
$\{(V^{i,\e}_t,W^{i,\e}_t)_{i=1,\dots,N,t\geq 0}, \e\in(0,1)\}$ is tight in $C([0,\infty),(\rd)^{2N})$.
We thus may consider a limit point $(V^{i}_t,W^{i}_t)_{i=1,\dots,N,t\geq 0}$ and we now check that
it satisfies all the requirements of the statement. Exchangeability, as well as points (i) and (ii)
(which concern only the initial conditions) are of course inherited from the fact that
they are satisfied by  $(V^{i,\e}_t,W^{i,\e}_t)_{i=1,\dots,N,t\geq 0}$ for all $\e\in(0,1)$.
Point (iv) is easily obtained by passing to the limit as $\e\to 0$ in \eqref{jab2}
(this uses only that $\sup_{\e\in(0,1)}\E[\sup_{[0,T]}\sum_1^N (|V^{i,\e}_t|^2+|W^{i,\e}_t|^2)]<\infty$).
Since $b,\sigma$ are continuous (and even Lipschitz continuous), it is not hard to pass
to the limit in \eqref{sys1} and to deduce that $(V^{i}_t)_{i=1,\dots,N,t\geq 0}$ is a weak solution to \eqref{ps},
whence $(V^{i}_t)_{i=1,\dots,N,t\geq 0}\sim \bL^N(F_N)$.
Using finally that $A(x,y)A^*(x,y)=a(y)=\sigma(y)\sigma^*(y)$ and that $A(-x,-y)=A(x,y)$,
we deduce that for each $\e\in(0,1)$, \eqref{sys2} can be rewritten
in the same form as \eqref{sys1} (with another family of Brownian motions $B^{ij}$). We thus prove as previously
that $(W^{i}_t)_{i=1,\dots,N,t\geq 0}\sim \bL^N(G_N)$ and this 
completes the proof. Observe that although $(V^{i}_t,W^{i}_t)_{i=1,\dots,N,t\geq 0}$ satisfies all the required
properties, it does not seem possible to check that it solves 
\eqref{sys1}-\eqref{sys2} with $\e=0$.
\end{proof}

\subsection{Rousset's inequality}\label{ssrou}

The following Lemma summarizes a few results found in \cite{r}.

\begin{lem}\label{ethop}
Consider $f,g\in \cP(\rd)$ and $R \in \cH(f,g)$ such that $\intrd |v|^2 f(dv)=\intrd |v|^2 g(dv)=1$,
$\intrd vf(dv)=\intrd vg(dv)=0$ and $\int_{\rd\times\rd} |v-w|^2 R(dv,dw) \leq 2$.
Denote by $\rho(f)$ the spectral radius of $\intrd v v^*f(dv) \in S_3^+$.
Observe that $\rho(f) \in (0,1]$, since $\intrd v v^*f(dv)$ has trace $1$.
For all $q>1$, there is a constant $C_q$ depending only on $q$ such that
$$
\int_{\rd\times\rd}  |v-w|^2 R(dv,dw) \leq C_q (1-\rho(f))^{-1} [m_{2+2q}(f+g)]^{1/q} [D(R)]^{1-1/q}.
$$
where $D(R)=\int_{\rd\times\rd}\int_{\rd\times\rd} [|v-x||w-y|-(v-x)\cdot(w-y)] R(dv,dw)R(dx,dy)$.
\end{lem}

We start with the following Lemma \cite[Theorem 1.4]{r}, of which we sketch the proof
for completeness.

\begin{lem}\label{specin}
Consider $f,g\in \cP(\rd)$ and $R \in \cH(f,g)$ such that $\intrd |v|^2 f(dv)=\intrd |v|^2 g(dv)=1$ and
$\intrd vf(dv)=\intrd vg(dv)=0$. For $\rho(f)$ the spectral radius of $\intrd v v^*f(dv)$, it holds that
\begin{align*}
D'(R) \geq& 2(1-\rho(f)) \Big(1-\Big[\int_{\rd\times\rd} (v\cdot w) R(dv,dw)\Big]^2 \Big).
\end{align*}
where $D'(R):=\int_{\rd\times\rd}\int_{\rd\times\rd} [|v-x|^2|w-y|^2-((v-x)\cdot(w-y))^2] R(dv,dw)R(dx,dy)$.
\end{lem}

\begin{proof}
Consider two independent couples $(U,V)$ and $(\tU,\tV)$ with law $R$. Using that
$\E[U]=\E[V]=0$ and $\E[|U|^2]=\E[|V|^2]=1$, a straightforward tedious computation shows that 
$$
D'(R)=\E[|U-\tU|^2|V-\tV|^2-[(U-\tU)\cdot(V-\tV)]^2]=2(A+B+C+D),
$$ 
where
$A= \E[|U|^2|V|^2-(U\cdot V)^2]$, $B=1-\E[U\cdot V]^2$, $C=\E[(U\cdot \tU)(V\cdot\tV)]-\E[(U\cdot \tV)^2]$
and $D=\E[(U\cdot \tU)(V\cdot\tV)]-\E[(U\cdot \tV)(\tU\cdot V)]$.
Clearly, $A\geq 0$ and it is not hard to verify that $D=\sum_{k,l=1}^3 (\E[U_kV_l]^2-\E[U_kV_l]\E[U_lV_k]) \geq 0$.
Next, $C=\sum_{k,l=1}^3  (\E[U_kV_l]^2-\E[U_kU_l]\E[V_kV_l])$. 
Working in an orthonormal basis in which 
$(\E[U_k U_l])_{k,l}$ is diagonal and using that $\rho(f)=\max_{k}\E[U_k^2]$,
$$
-C\leq \sum_{k=1}^3 (\E[U_k^2]\E[V_k^2]-\E[U_kV_k]^2) \leq \rho(f)  \sum_{k=1}^3 \Big(\E[V_k^2]-\frac{\E[U_kV_k]^2}
{\E[U_k^2]}\Big) =  \rho(f)\Big(1- \sum_{k=1}^3\frac{\E[U_kV_k]^2}
{\E[U_k^2]}\Big)
$$ 
because $\E[|U|^2]=1$. But by Cauchy-Scwharz's inequality (and since, again, $\E[|U|^2]=1$),
$$
\E[U\cdot V]^2=\Big(\sum_{k=1}^3 \E[U_kV_k]\Big)^2 \leq \Big(\sum_{k=1}^3\frac{\E[U_kV_k]^2}{\E[U_k^2]}\Big)
\Big(\sum_{k=1}^3 \E[U_k^2]\Big)=\sum_{k=1}^3\frac{\E[U_kV_k]^2}{\E[U_k^2]}.
$$
Finally, $-C \leq \rho(f)(1-\E[U\cdot V]^2)$ and $D'(R) \geq 2B+2C\geq 2(1-\rho(f))(1-\E[U\cdot V]^2)$.
\end{proof}

\begin{proof}[Proof of Lemma \ref{ethop}]
Using $\intrd |v|^2 f(dv)=\intrd |v|^2 g(dv)=1$ and
$\int_{\rd\times\rd} |v-w|^2 R(dv,dw) \leq 2$, we deduce that $0\leq \int_{\rd\times\rd} (v\cdot w) R(dv,dw)\leq 1$
and then that 
$$
\int_{\rd\times\rd} |v-w|^2 R(dv,dw)=2-2\int_{\rd\times\rd} (v\cdot w) R(dv,dw)
\leq 2-2\Big(\int_{\rd\times\rd} (v\cdot w) R(dv,dw)\Big)^2.
$$
Applying next Lemma \ref{specin}, we find
$$
\int_{\rd\times\rd} |v-w|^2 R(dv,dw) \leq  (1-\rho(f))^{-1} D'(R).
$$
Using that (recall that $q>1$)
$$
|X|^2|Y|^2-(X\cdot Y)^2\leq [|X||Y|-(X\cdot Y)][2|X||Y|]\leq
[|X||Y|-(X\cdot Y)]^{1-1/q}[2|X||Y|]^{1+1/q}
$$
and the H\"older inequality, we see that $D'(R)\leq D(R)^{1-1/q} (K_q(R))^{1/q}$, where we have set
$K_q(R)=\int_{\rd\times\rd}\int_{\rd\times\rd}  [2|v-x||w-y|]^{q+1}R(dv,dw)R(dx,dy)$. To conclude,
it suffices to observe that $K_q(R) \leq C_q m_{2q+2}(f+g)$, which immediately follows
from the fact that $R \in \cH(f,g)$.
\end{proof}

\subsection{Conclusion}\label{ssconc}

We can now give the

\begin{proof}[Proof of Theorem \ref{rouss}.]
We fix $N\geq 7$, $p\in [0, (N-6)/2]$ and some exchangeable 
$F_N \in \cP(S_N)$. We put $q=p+1$. We apply Lemma \ref{coup} with $G_N=\cU(S_N)$
to build a coupling $(V^{i,N}_t,W^{i,N}_t)_{i=1,\dots,N,t\geq 0}$ between $\bL^N(F_N)$ and $\bL^N(\cU(S_N))$.
We introduce the notation $U^N_t=N^{-1}\sum_1^N |V^{i,N}_t-W^{i,N}_t|^2$ and $u^N_t=\E[U^N_t]$. We also set
$\mu^N_t=N^{-1}\sum_1^N \delta_{V^{i,N}_t}$, $\nu^N_t=N^{-1}\sum_1^N \delta_{W^{i,N}_t}$,
as well as $\zeta^N_t=N^{-1}\sum_1^N \delta_{(V^{i,N}_t,W^{i,N}_t)}$.
Lemma \ref{coup}-(iv) precisely says that $U^N_t=U^N_0-2\intot D(\zeta^N_s)ds$,
with $D$ defined in Lemma \ref{ethop}. 
Since $U^N_0\leq 2$ by Lemma \ref{coup}-(ii),
we deduce that a.s., $U^N_t\leq 2$ for all $t\geq 0$. 

\vip

We now apply Lemma \ref{ethop} with $R=\zeta^N_t \in \cH(\mu^N_t,\nu^N_t)$,
which is licit since $\intrd v \mu^N_t(dv)=\intrd v \nu^N_t(dv)=0$ and $\intrd |v|^2 \mu^N_t(dv)=
\intrd |v|^2 \nu^N_t(dv)=1$ (because both $F_N$ and $G_N$ are carried by $S_N$) and since
$U^N_t= \int_{\rd\times\rd} |v-w|^2 \zeta^N_t(dv,dw)\leq 2$: we deduce that
$$
U^N_t\leq C_q (1-\rho(\nu^N_t))^{-1} [m_{2+2q}(\mu^N_t+\nu^N_t)]^{1/q} [D(\zeta^N_t)]^{1-1/q}.
$$
Taking expectations and using the triple H\"older inequality (with $2q$, $2q$ and $q/(q-1)$),
\begin{align*}
u^N_t \leq C_q \E[(1-\rho(\nu^N_t))^{-2q}]^{1/(2q)}\E[(m_{2+2q}(\mu^N_t+\nu^N_t))^2]^{1/(2q)}\E[D(\zeta^N_t)]^{1-1/q}.
\end{align*}
Since $(W^{i,N}_t)_{i=1,\dots,N}\sim\cU(S_N)$ for each $t\geq 0$ by Remark \ref{equi}, 
we infer from Lemma \ref{ums}-(ii)
that $\E[(m_{2+2q}(\nu^N_t))^2]\leq \E[|W^{1,N}_t|^{4q+4}]\leq C_q$ and from  Lemma \ref{ums}-(iii),
since $\rho(\nu^N_t)$ is the spectral radius of $\intrd vv^* \nu^N_t(dv)=N^{-1}\sum_1^N W^{i,N}_t(W^{i,N}_t)^*$
and since $2q \leq N-4$, that
$\E[(1-\rho(\nu^N_t))^{-2q}]\leq C_q$. Also, we know from Proposition \ref{mps}
that $\E[(m_{2+2q}(\mu^N_t))^2]\leq \E[|V^{1,N}_t|^{4q+4}] \leq C_q\E[|V^{1,N}_0|^{4q+4}]$.
We end with
$$
u^N_t \leq C_q \E[1+|V^{1,N}_0|^{4q+4}]^{1/(2q)}\E[D(\zeta^N_t)]^{1-1/q}.
$$
Recalling that $U^N_t=U^N_0-2\intot D(\zeta^N_s)ds$, we conclude that, for some $c_q>0$
depending only on $q$,
$$
\frac d{dt}u^N_t = -2\E[D(\zeta^N_t)] \leq -c_q \E[1+|V^{1,N}_0|^{4q+4}]^{-1/(2(q-1))} (u^N_t)^{q/(q-1)}.
$$
Integrating this inequality, we find, recalling that $p=q-1$ and setting $\kappa_p=c_q/(q-1)$,
$$
u^N_t \leq \Big(\kappa_p \E[1+|V^{1,N}_0|^{8+4p}]^{-1/(2p)} t +  (u^N_0)^{-1/p}\Big)^{-p}.
$$
By construction, since $\bL^N_t(\cU(S_N))=\cU(S_N)$ for all $t\geq 0$, 
we have $N^{-1}\cW_2^2(\bL^N_t(F_N),\cU(S_N))\leq u^N_t$ and, by Lemma
\ref{coup}-(i)-(ii), $u^N_0=N^{-1}\cW_2^2(F_N,\cU(S_N))\leq 2$. We conclude that
$$
N^{-1}\cW_2^2(\bL^N_t(F_N),\cU(S_N)) \leq \Big(\kappa_p \E[1+|V^{1,N}_0|^{8+4p}]^{-1/(2p)} t +  
(N^{-1}\cW_2^2(F_N,\cU(S_N)))^{-1/p}\Big)^{-p}.
$$
On the onde hand, this implies that $N^{-1}\cW_2^2(\bL^N_t(F_N),\cU(S_N)) \leq N^{-1}\cW_2^2(F_N,\cU(S_N)$.
On the other hand, this gives  $N^{-1}\cW_2^2(\bL^N_t(F_N),\cU(S_N))\leq 
(\kappa_p \E[1+|V^{1,N}_0|^{8+4p}]^{-1/(2p)} t +  2^{-1/p})^{-p}$, which is controlled by
$C_p \E[1+|V^{1,N}_0|^{8+4p}]^{1/2}(1+t)^{-p}$.
\end{proof}

\subsection{Uniform propagation of chaos}\label{ssff}
We start with a consequence of Theorem \ref{rouss}.

\begin{cor}\label{corrouss}
Assume that $\gamma=0$, fix $N\geq 2$ and consider some exchangeable $S_N$-valued initial
condition $(V^{i,N}_0)_{i=1,\dots,N}$, the corresponding solution $(V^{i,N}_t)_{i=1,\dots,N,t\geq 0}$ to \eqref{ps}
and set $\mu^N_t=N^{-1}\sum_1^N\delta_{V^{i,N}_t}$, 
For all $p>0$, there is a constant $C_p$ depending only on $p$ such that for all $t\geq 0$,
\begin{align*}
\E[\cW_2^2(\mu^N_t,\cN(0,3^{-1}\Id))] \leq C_p \Big( N^{-1/2}+ \E[1+|V^{1,N}_0|^{8+4p}]^{1/2}  (1+t)^{-p}\Big).
\end{align*}
\end{cor}

\begin{proof}
Let $p>0$ be fixed. If first $N-6<2p$, then we simply use that
$\cW_2^2(\mu^N_t,\cN(0,3^{-1}\Id))\leq 2$ a.s., so that the inequality obviously holds with
$C_p= 2\sqrt{2p+6}$.

\vip

If next $N-6\geq 2p$, we use Theorem \ref{rouss}: for all $t\geq 0$, 
there is $(X^{i,N}_t)_{i=1,\dots,N}\sim\cU(S_N)$
such that $N^{-1}\sum_1^N\E[|V^{i,N}_t-X^{i,N}_t|^2] \leq C_p \E[1+|V^{1,N}_0|^{8+4p}]^{1/2}(1+t)^{-p}$. 
We now put $\nu^N_t=N^{-1}\sum_1^N
\delta_{X^{i,N}_t}$ and we know from Lemma \ref{ums} that $\E[\cW_2^2(\nu^N_t,\cN(0,3^{-1}\Id))] \leq CN^{-1/2}$.
But $\cW_2^2(\mu^N_t,\nu^N_t)\leq N^{-1}\sum_1^N|V^{i,N}_t-X^{i,N}_t|^2$,
whence $\E[\cW_2^2(\mu^N_t,\nu^N_t)]\leq C_p \E[1+|V^{1,N}_0|^{8+4p}]^{1/2}(1+t)^{-p}$.
\end{proof}

We finally give the

\begin{proof}[Proof of Theorem \ref{mr2}-(i)]
Recall that $\gamma=0$, that $f_0\in\cP_2(\rd)$
and that $(f_t)_{t\geq 0}$ is the unique weak solution to \eqref{HL3D}. We assume without loss of
generality that $\intrd v f_0(dv)=0$ and that $m_2(f_0)=1$. 
For each $N\geq 2$, we consider an exchangeable $(\rd)^N$-valued random variable $(V^{i,N}_0)_{i=1,\dots,N}$
and the solution $(V_t^{i,N})_{i=1,\dots,N,t\geq 0}$ to \eqref{ps}. We set 
$\mu^N_t=N^{-1}\sum_1^N \delta_{V^{i,N}_t}$. We assume that for all $p\geq 2$,
$M_p:=m_p(f_0)+\sup_{N\geq 2}\E[|V^{1,N}_0|^p]<\infty$.
The constants are allowed to depend on upperbounds of $\{M_p,p\geq 2\}$
and on some upperbound of $H(f_0)$ when it is finite.
We fix $\eta\in(0,1/5)$.

\vip

{\it Step 1.} By Theorem \ref{maxweak}, we have
\vip
(i) $\E[\cW_2^2(\mu^N_t,f_t)]\leq C_\eta(1+t)^{5/2}(\E[\cW_2^2(\mu^N_0,f_0)]+N^{\eta-1/4})$ in general;
\vip
(ii) $\E[\cW_2^2(\mu^N_t,f_t)]\leq C_\eta(1+t)^{5/2}(\E[\cW_2^2(\mu^N_0,f_0)]+N^{\eta-1/3})$
if $H(f_0)<\infty$.

\vip

{\it Step 2.} Here we verify that for any $p>0$, 
$$
\E[\cW_2^2(\mu^N_t,\cN(0,3^{-1}\Id))]\leq C_p (1+t)^{-p} + C_p N^{-1/2} + C \E[\cW_2^2(\mu^N_0,f_0)].
$$
We put $m^N_0=N^{-1}\sum_1^N V^{i,N}_0$ and $E^N_0=N^{-1}\sum_1^N |V^{i,N}_0-m^N_0|^2$.
On the event $\Omega_N=\{E^N_0 \geq 1/4\}$, we set
$\hat V^{i,N}_t=(V^{i,N}_t-m^N_0)/\sqrt{E^N_0}$ and $\hat\mu^N_t=N^{-1}\sum_1^N \delta_{\hat V^{i,N}_t}$.
Conditionally on $\Omega_N$, $(\hat V^{i,N}_0)_{i=1,\dots,N}$ is exchangeable and takes values in $S_N$,
so that we can apply Corollary \ref{corrouss}: 
$$
\E[\indiq_{\Omega_N}\cW_2^2(\hat\mu^N_t,\cN(0,3^{-1}\Id))] \leq \frac {C_p} {\sqrt N}+
\frac{C_p\E[1+\indiq_{\Omega_N}|\hat V^{1,N}_0|^{8+4p}]^{1/2}}{(1+t)^{p}}\leq  \frac {C_p} {\sqrt N}+ \frac{C_p}{(1+t)^{p}}.
$$
For the last inequality, we used that $|\hat V^{1,N}_0|\leq 4|V^{1,N}_0|+4|m^N_0|$ on $\Omega_N$, whence
$\E[\indiq_{\Omega_N}|\hat V^{1,N}_0|^{8+4p}]\leq C_p\E[|V^{1,N}_0|^{8p+4}]
+C_p\E[|m^N_0|^{8p+4}]\leq C_p \E[|V^{1,N}_0|^{8p+4}]$, which is bounded by assumption.

\vip

Next, we write 
$$
\cW_2^2(\hat \mu^N_t,\mu^N_t)\!\leq\! \frac1N\sum_1^N |V^{i,N}_0-\hat V^{i,N}_0|^2 \!=\!
\frac1N\sum_1^N \Big|(V^{i,N}_t\!\!-m^N_0)\frac{\sqrt{E^N_0}-1}{\sqrt{E^N_0}} + m^N_0\Big|^2\!
=\! \Big(\sqrt{E^N_0}-1\Big)^2
+ |m^N_0|^2.
$$
By Remark \ref{rkd}, we have  $(\sqrt{E^N_0}-1)^2+ |m^N_0|^2\leq\cW_2^2(\mu^N_0,f_0)$,
because $m_{\mu^N_0}=m^N_0$, $V_{\mu^N_0}=E^N_0$, $m_{f_0}=0$ and $V_{f_0}=1$.
At this point, we have proved that 
$$
\E[\indiq_{\Omega_N}\cW_2^2(\mu^N_t,\cN(0,3^{-1}\Id))] \leq
C_p N^{-1/2} + C_p(1+t)^{-p} + 2\E[\cW_2^2(\mu^N_0,f_0)].
$$

We next observe that 
$\E[\indiq_{\Omega_N^c}\cW_2^2(\mu^N_t,\cN(0,3^{-1}\Id))] \leq \E[\indiq_{\{E^N_0<1/4\}} m_2(\mu^N_t+\cN(0,3^{-1}\Id))]
=\E[\indiq_{\{E^N_0<1/4\}} (E^N_0+|m^N_0|^2+1)] \leq (5/4) \Pr(E^N_0<1/2)+\E[|m^N_0|^2]$.
But, $\Pr(E^N_0<1/4)\leq \Pr(|\sqrt{E^N_0}-1|>1/2)
\leq 4\E[|\sqrt{E^N_0}-1|^2]$, and we have checked that $\E[\indiq_{\Omega_N^c}\cW_2^2(\mu^N_t,\cN(0,3^{-1}\Id))]
\leq 5 \E[|\sqrt{E^N_0}-1|^2]+\E[|m^N_0|^2]$, which is controlled by  $5\E[\cW_2^2(\mu^N_0,f_0)]$
as seen a few lines above.

\vip

{\it Step 3.} We deduce that $\cW_2^2(f_t,\cN(0,3^{-1}\Id)) \leq C_p (1+t)^{-p}$:
assume (only during this step) that $(V^{i,N}_0)_{i=1,\dots,N}$ consists of i.i.d. $f_0$-distributed random variables,
so that $\E[\cW_2^2(\mu^N_0,f_0)] \leq C N^{-1/2}$ by \cite[Theorem 1]{fgui}. Write
$\cW_2^2(f_t,\cN(0,3^{-1}\Id))\leq 2 \E[\cW_2^2(f_t,\mu^N_t)]+2\E[ \cW_2^2(\mu^N_t,\cN(0,3^{-1}\Id))]
\leq C_\eta(1+t)^{5/2}( N^{-1/2}+N^{\eta-1/4}) + C_p (1+t)^{-p} + C N^{-1/2}$ by Steps 1 and 2. 
It then suffices to let $N\to\infty$.

\vip

{\it Step 4.} We now conclude the proof in the general case.

\vip

If $(1+t)^{5/2}\leq (\E[\cW_2^2(\mu^N_0,f_0)]+N^{-1/4})^{-\eta}$,
then we use Step 1-(i) to write $\E[\cW_2^2(\mu^N_t,f_t)]\leq C_\eta(\E[\cW_2^2(\mu^N_0,f_0)]+N^{-1/4})^{-\eta}
(\E[\cW_2^2(\mu^N_0,f_0)]+N^{\eta-1/4})\leq C_\eta (\E[\cW_2^2(\mu^N_0,f_0)]+N^{-1/4})^{1-5\eta}$.

\vip

If now  $(1+t)^{5/2}> (\E[\cW_2^2(\mu^N_0,f_0)]+N^{-1/4})^{-\eta}$,
then we use Steps 2 and 3 with $p=5/(2\eta)$ to write 
$\E[\cW_2^2(\mu^N_t,f_t)]\leq C_\eta(1+t)^{-5/(2\eta)}+ C_\eta N^{-1/2}+ C\E[\cW_2^2(\mu^N_0,f_0)]$.
But $(1+t)^{-5/(2\eta)} \leq \E[\cW_2^2(\mu^N_0,f_0)]+N^{-1/4}$ and we end with 
$\E[\cW_2^2(\mu^N_t,f_t)]\leq C_\eta (\E[\cW_2^2(\mu^N_0,f_0)]+N^{-1/4})$.

\vip

Thus $\sup_{[0,\infty)}\E[\cW_2^2(\mu^N_t,f_t)]\leq 
C_\eta (\E[\cW_2^2(\mu^N_0,f_0)]+N^{-1/4})^{1-5\eta}$ as desired.

\vip

{\it Step 5.} We finally conclude when $H(f_0)<\infty$.

\vip

If $(1+t)^{5/2}\leq (\E[\cW_2^2(\mu^N_0,f_0)]+N^{-1/3})^{-\eta}$,
then we use Step 1 to write $\E[\cW_2^2(\mu^N_t,f_t)]\leq C_\eta(\E[\cW_2^2(\mu^N_0,f_0)]+N^{-1/3})^{-\eta}
(\E[\cW_2^2(\mu^N_0,f_0)]+N^{\eta-1/3})\leq C_\eta (\E[\cW_2^2(\mu^N_0,f_0)]+N^{-1/3})^{1-4\eta}$.

\vip

If now  $(1+t)^{5/2}> (\E[\cW_2^2(\mu^N_0,f_0)]+N^{-1/3})^{-\eta}$,
then we use Steps 2 and 3 with $p=5/(2\eta)$ to write 
$\E[\cW_2^2(\mu^N_t,f_t)]\leq C_\eta(1+t)^{-5/(2\eta)}+ C_\eta N^{-1/2}+ C\E[\cW_2^2(\mu^N_0,f_0)]$.
But $(1+t)^{-5/(2\eta)} \leq \E[\cW_2^2(\mu^N_0,f_0)]+N^{-1/3}$ and we end with
$\E[\cW_2^2(\mu^N_t,f_t)]\leq C_\eta (\E[\cW_2^2(\mu^N_0,f_0)]+N^{-1/3})$.

\vip

We conclude that $\sup_{[0,\infty)}\E[\cW_2^2(\mu^N_t,f_t)]\leq 
C_\eta (\E[\cW_2^2(\mu^N_0,f_0)]+N^{-1/3})^{1-4\eta}$ as desired.
\end{proof}

\end{document}